\documentclass[a4paper]{amsart}

\usepackage{amssymb} 
\usepackage[all]{xy}
\usepackage{dsfont}
\usepackage{hyperref}

\SelectTips{cm}{}
\calclayout 

\thanks{Version from 26th April 2012}

\numberwithin{equation}{section}

\theoremstyle{plain}
\newtheorem{theorem}[equation]{Theorem}
\newtheorem{proposition}[equation]{Proposition}
\newtheorem{lemma}[equation]{Lemma} 
\newtheorem{corollary}[equation]{Corollary}

\theoremstyle{definition}

\newtheorem{example}[equation]{Example}

\theoremstyle{remark} 
\newtheorem{remark}[equation]{Remark} 

\newcommand{\Abel}{\operatorname{\mathsf{Ab}}}
\newcommand{\Acy}{\operatorname{\mathbf A}}
\newcommand{\acy}{\operatorname{A}}

\newcommand{\cosupp}{\operatorname{cosupp}}
\newcommand{\Dl}{\operatorname{\mathbf D}}

\newcommand{\End}{\operatorname{End}}

\newcommand{\fHom}{\operatorname{\mathcal{H}\!\!\;\mathit{om}}}

\newcommand{\Hom}{\operatorname{Hom}}
\newcommand{\id}{\operatorname{id}}
\newcommand{\Id}{\operatorname{Id}}

\newcommand{\Inj}{\operatorname{\mathsf{Inj}}}
\newcommand{\Ker}{\operatorname{Ker}}
\newcommand{\KInj}[1]{\mathsf K(\Inj #1)}

\newcommand{\Tloc}{\operatorname{\mathbf{Loc}}}
\newcommand{\tloc}{\operatorname{Loc}}

\newcommand{\Qcoh}{\operatorname{Qcoh}}

\newcommand{\Sp}{\operatorname{Sp}}
\newcommand{\Spec}{\operatorname{Spec}}
\newcommand{\StMod}{\operatorname{\mathsf{StMod}}}

\newcommand{\supp}{\operatorname{supp}}

\newcommand{\Th}{\operatorname{\mathbf{Th}}}
\newcommand{\Thick}{\operatorname{Th}}

\newcommand{\op}{\mathrm{op}}
\newcommand{\per}{\mathrm{per}}

\newcommand{\col}{\colon}

\newcommand{\ges}{{\scriptscriptstyle\geqslant}}

\newcommand{\kos}[2]{{#1}/\!\!/{#2}}

\newcommand{\lto}{\longrightarrow}
\renewcommand{\setminus}{\smallsetminus}

\newcommand{\xra}{\xrightarrow}

\newcommand{\bik}{Benson/Iyengar/Krause}

\def\mcH{\mathcal{H}}

\def\sfT{\mathcal{T}}
\def\mcU{\mathcal{U}} 
\def\mcV{\mathcal{V}}

\def\sfc{\mathsf c}

\def\sfA{\mathsf A} 
\def\Ab{\mathsf{Ab}}
 
\def\sfC{\mathsf C}
\def\sfD{\mathsf D}

\def\sfP{\mathsf P}
\def\sfQ{\mathsf Q}

\def\sfS{\mathsf S} 
\def\Sets{\mathsf{Sets}} 
\def\sfT{\mathsf T} 
\def\sfU{\mathsf U}

\def\bbZ{\mathbb Z}

\newcommand{\fp}{\mathfrak{p}}
\newcommand{\fq}{\mathfrak{q}}

\newcommand{\gam}{\varGamma} 
\newcommand{\lam}{\varLambda}

\def\Si{\Sigma} 
\def\si{\sigma}

\def\one{\mathds 1}

\title[The Bousfield lattice and stratification]
{The Bousfield lattice of a triangulated category and stratification}

\author{Srikanth B. Iyengar} 
\address{Srikanth B. Iyengar\\ 
Department of Mathematics\\ 
University of Nebraska\\ 
Lincoln, NE 68588\\ 
U.S.A.}

\author{Henning Krause} 
\address{Henning Krause\\ 
Fakult\"at f\"ur Mathematik\\ 
Universit\"at Bielefeld\\ 
33501 Bielefeld\\ 
Germany.}

\begin{document}

\begin{abstract}
  For a tensor triangulated category which is well generated in the
  sense of Neeman, it is shown that the collection of Bousfield
  classes forms a set. This set has a natural structure of a complete
  lattice which is then studied, using the notions of
  stratification and support.
\end{abstract}



\thanks{The research of the first author was partly supported by NSF grant DMS 0903493. }

\maketitle 

\setcounter{tocdepth}{1} 
\tableofcontents

\section{Introduction}

In recent work with Benson~\cite{\bik:2008a,\bik:2008b,\bik:2009a}, we introduced a notion of stratification of a triangulated category via
the action of a graded commutative ring, and used it to classify the localizing subcategories of the stable module category of a finite
group. In this paper, we relate these ideas to work on the lattice of the stable homotopy category of spectra initiated by Bousfield~\cite{Bousfield:1979a}, and through this, also to the theory of support being developed by Balmer~\cite{Balmer:2005a}. Tensor
triangulated categories, meaning triangulated categories admitting set-indexed coproducts and carrying a tensor product compatible with
the triangulated structure, are an appropriate framework for our analysis.  The stable homotopy category of spectra and the stable
module category of a finite group are important examples, but there are many more.

Let $(\sfT,\otimes,\one)$ be a tensor triangulated category. Mimicking
\cite{Bousfield:1979a}, for each object $X$ in $\sfT$ the class of
objects annihilated by the functor $X\otimes-$ is said to be the
\emph{Bousfield class} of $X$. Objects $X$ and $Y$ are \emph{Bousfield
  equivalent} if their Bousfield classes coincide. Our first result,
Theorem~\ref{th:cardinality}, involves the notion of well generated
triangulated categories, introduced by Neeman~\cite{Neeman:2001a}, and
asserts:

\smallskip

\noindent\emph{In a well generated tensor triangulated category, the Bousfield classes form a set.}

\smallskip

This result was proved by Ohkawa~\cite{Ohkawa:1989a} for the stable
homotopy category, but seems to be new even for compactly generated
triangulated categories; it gives an affirmative answer to
Question~5.9 in \cite{Dwyer/Palmieri:2008a}, which asked if the
Bousfield classes in the derived category of any commutative ring form
a set.

The set of Bousfield classes has a structure of a complete lattice,
called the \emph{Bousfield lattice}.  It has been studied by
topologist for the stable homotopy category of spectra, see for
example \cite{Bousfield:1979a,Ravenel:1992a,Hovey/Palmieri:1999a,Strickland:2004a}, but
not much is known for general tensor triangulated categories. In this
work we combine a general analysis of the Bousfield lattice with an
analysis for some specific classes of tensor triangulated categories.

In Sections~\ref{se:stratification} and \ref{sec:Local objects}, we
concentrate on compactly generated tensor triangulated categories that
are \emph{stratified}, in the sense of \cite{\bik:2008b}, by an action of a graded-commutative noetherian
ring. For instance, in Corollary~\ref{cor:locset},
we give a complete description of the Bousfield lattice of such a
tensor triangulated category in terms of the prime ideal spectrum of
the ring. This connection was first made by Hovey, Palmieri, and
Strickland \cite{Hovey/Palmieri/Strickland:1997a} in their work on
axiomatic stable homotopy theory.

A new ingredient in our work is a connection with support
and cosupport for objects in $\sfT$, introduced in \cite{\bik:2008a}
and \cite{\bik:2010a}, respectively. For instance, the stratification
condition implies that the tensor product formula
\[
\supp_R(X\otimes Y)=\supp_R(X)\cap\supp_R(Y)
\]
holds for all $X,Y$ in $\sfT$; see \cite{\bik:2009a}. Since $\supp_{R}(X)=\varnothing$ only when $X=0$, the
Bousfield class $\acy(X)$ of an object $X$ has thus the form
\[
\acy(X) = \{Y\in\sfT\mid\supp_R(Y)\cap\supp_R(X)=\varnothing\}\,,
\] 
while the $X$-local objects can be described as follows:
\[
\acy(X) ^\perp= \{Y\in\sfT\mid\cosupp_R(Y)\subseteq\supp_R(X)\}\,.
\] 
It follows that an object $X$ in $\sfT$ is $X$-local precisely when $\cosupp_{R}(X)\subseteq\supp_{R}(X)$; this
inclusion holds when $X$ is compact, but not in general.

In Sections~\ref{se:Stone} and \ref{sec:Compact objects} of this
paper, we employ concepts from lattice theory and introduce a notion
of support which uses the Bousfield lattice. This is an attempt to
develop a theory of support for tensor triangulated categories
independent of any action of a ring. A typical example is the derived
category of the category of quasi-coherent sheaves on a noetherian
scheme. For other work in that direction, see Balmer and Favi
\cite{BF}, and Stevenson \cite{Stevenson:2011a}.
 
The support defined in terms of the Bousfield lattice is closely
related to Balmer's theory of support \cite{Balmer:2005a}. The
following diagram makes this more precise, and its commutativity will
be a consequence of our analysis.
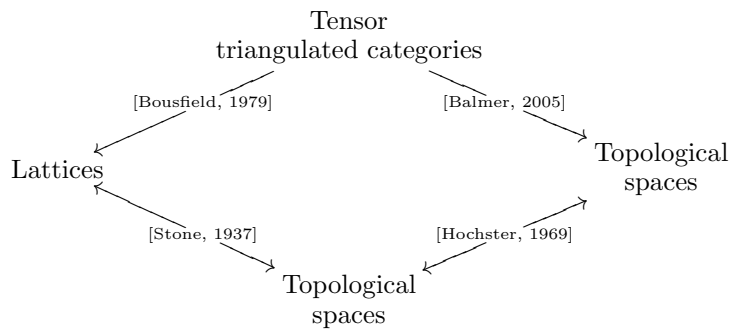
\[
\xymatrixrowsep{2pc} \xymatrixcolsep{3pc} \xymatrix{ &*+\txt{Tensor\\
    triangulated categories}\ar[ld]|{\txt{\tiny{[Bousfield,
        1979]}}}\ar[rd]|{\txt{\tiny{[Balmer, 2005]}}}\\
  *+\txt{Lattices}\ar@{<->}[rd]|{\txt{\tiny{[Stone, 1937]}}}
  &&*+\txt{Topological \\spaces} \ar@{<->}[ld]|{\txt{\tiny{[Hochster,
        1969]}}} \\
  &*+\txt{Topological \\spaces} }
\] 
We hasten to add that one needs to restrict at each vertex to some
appropriate class of objects to make sure that all arrows are
well-defined. On the left, one composes the construction of the
Bousfield lattice with Stone duality \cite{Stone:1937a}, which means
that a distributive lattice is represented via its associated
spectrum.

In \cite{Balmer:2005a}, Balmer associates with any essentially small
tensor triangulated category a spectrum of prime ideals together with
a Zariski topology. In \cite{Buan/Krause/Solberg:2007a}, it is shown
that this space is spectral in the sense of Hochster
\cite{Hochster:1969a}. There is a corresponding Hochster dual space
and that completes the right half of the above diagram. There is also
an objectwise interpretation of this diagram, assigning to any object
of a tensor triangulated category its Bousfield class on the left and
its support on the right. The commutativity of the above diagram then
amounts to the following assertion which is a consequence of
Theorem~\ref{th:spec}:

\smallskip
\noindent 
\emph{The Hochster dual of Balmer's spectrum of the category of
  compact objects in $\sfT$ is homeomorphic to the spectrum associated
  with the sublattice of the Bousfield lattice of $\sfT$ generated by
  the compact objects. This homeomorphism identifies the support of a
  compact object $X$, which is a Zariski closed subset, with the
  quasi-compact open subset in the Stone topology corresponding to the
  Bousfield class of $X$.}

\smallskip

\subsection*{Acknowledgments} 
We learned about the Stone duality point of view on spectra and supports from a talk by Joachim Kock in 2007,
as explained in \cite{Kock:2007}. It is a pleasure to thank him, and also Ivo Dell'Ambrogio and Greg Stevenson for
inspiring discussions and helpful comments on this work. 

A preliminary version of this work was presented at a meeting at the Mathematisches
Forschungsinstitut at Oberwolfach, see~\cite{Iyengar:2011}. We thank
the institute for providing an opportunity for research
collaboration. The second author gratefully acknowledges support from
the Hausdorff Research Institute for Mathematics at Bonn where part of
this work was done.

\section{Localizing subcategories}

In this section we discuss a hierarchy of localizing subcategories. We
work in the context of well generated triangulated categories in the
sense of Neeman \cite{Neeman:2001a}. This includes the class of
compactly generated categories. Note that localizing subcategories or
Verdier quotient categories of a compactly generated triangulated
category are usually not compactly generated; but they are often well
generated. Thus the class of well generated triangulated categories
seems to be the appropriate universe for studying triangulated
categories.

\subsection*{Well generated triangulated categories}
Let $\sfT$ be a \emph{well generated triangulated category}. Thus
$\sfT$ is a triangulated category with set-indexed coproducts and
there is a set of perfect generators which are $\alpha$-small for some
regular cardinal $\alpha$; see \cite{Neeman:2001a,Krause:2001a} for
details. In that case one calls $\sfT$ \emph{$\alpha$-well generated}.

Now suppose that $\sfT$ is $\alpha$-well generated, and therefore
$\beta$-well generated for every regular cardinal $\beta\ge\alpha$. We
denote by $\sfT^\alpha$ the full subcategory of \emph{$\alpha$-compact
  objects} in $\sfT$. This is by definition the smallest triangulated
subcategory of $\sfT$ which contains a set of $\alpha$-small perfect
generators of $\sfT$ and is closed under \emph{$\alpha$-coproducts},
that is, coproducts with less than $\alpha$ factors. Note that
$\sfT^\alpha$ does not depend on the choice of perfect generators of
$\sfT$ \cite[Lemma~5]{Krause:2001a}. Moreover,
$\sfT=\bigcup_{\beta\ges \alpha}\sfT^{\beta}$, where $\beta$ runs
through all regular cardinals.

If $\alpha=\aleph_0$, then $\sfT$ is called \emph{compactly
generated}. In that case $\sfT^{\aleph_0}$ coincides with the
subcategory of compact objects and one writes $\sfT^c$ for
$\sfT^{\aleph_0}$.

Denote by $\Abel_\alpha(\sfT)$ the category of additive functors
$(\sfT^\alpha)^\op\to\Ab$ into the category of abelian groups which
send $\alpha$-coproducts in $\sfT^\alpha$ to products in $\Ab$. This
is an abelian category, and the \emph{restricted Yoneda functor}
\[ H_\alpha\col\sfT\longrightarrow \Abel_\alpha(\sfT),\quad
X\longmapsto \Hom_\sfT(-,X)|_{\sfT^\alpha}
\] is the universal cohomological and coproduct preserving functor
into an abelian category which has set-indexed coproducts and exact
$\alpha$-filtered colimits \cite[Proposition~6.10.1]{Krause:2010a}. If
$\alpha=\aleph_0$, then $H_\alpha$ is the universal functor into an
abelian category satisfying Grothendieck's AB5 condition; see
\cite[Corollary~2.4]{Krause:2000}, or
\cite[Proposition~7.3]{Christensen/Strickland:1998a} for the stable
homotopy category. For example, any module category is an AB5
category.

\subsection*{Localizing subcategories}
A full subcategory of $\sfT$ is called \emph{localizing} if it is a
triangulated subcategory which is closed under set-indexed
coproducts. A \emph{localization functor} $L\col\sfT\to \sfT$ is an
exact functor that admits a natural transformation $\eta\col
\Id_{\sfT}\to L$ such that $L(\eta X)$ is an isomorphism and $L(\eta
X)=\eta(LX)$ for all objects $X$ in $\sfT$.

\begin{proposition}
\label{prop:loc} Let $\sfT$ be an $\alpha$-well generated triangulated
category, where $\alpha$ denotes a fixed regular cardinal.  Consider
for a localizing subcategory $\sfS$ of $\sfT$ the following
statements:
\begin{enumerate}
\item There is an exact functor $F\col \sfT\to\sfT$ preserving
set-indexed coproducts such that $\sfS=\Ker F$.
\item There is a cohomological functor $H\col \sfT\to\sfA$ preserving
set-indexed coproducts into an abelian category which has set-indexed
coproducts and exact $\alpha$-filtered colimits such that $\sfS=\Ker
H$.
\item There is a cohomological functor $H\col \sfT\to\sfA$ preserving
set-indexed coproducts into an abelian category which has set-indexed
coproducts and exact $\beta$-filtered colimits for some regular
cardinal $\beta$ such that $\sfS=\Ker H$.
\item There is an object $X$ in $\sfT$ such that $\sfS$ equals the
smallest localizing subcategory of $\sfT$ containing $X$.
\item There is a localization functor $L\col \sfT\to\sfT$ such that
$\sfS=\Ker L$.
\end{enumerate} Then the implications $(1) \implies (2) \implies (3)
\iff (4) \implies (5)$ hold.
\end{proposition}

\begin{proof} (1) $\Rightarrow$ (2): Let $H$ be the composite of $F$
with the restricted Yoneda functor $\sfT\to \Abel_\alpha(\sfT)$. Then
$\Ker H=\sfS$.

(2) $\Rightarrow$ (3): Clear.

(3) $\Rightarrow$ (4): In \cite[Theorem~7.5.1]{Krause:2010a} it is
shown that $\sfS=\Ker H$ is well generated. Now take for $X$ the
coproduct of a set of generators of $\sfS$.

(4) $\Rightarrow$ (3): The Verdier quotient $\sfT/\sfS$ is
$\beta$-well generated for some regular cardinal $\beta$, by
\cite[Corollary 4.4.3]{Neeman:2001a} or
\cite[Theorem~7.2.1]{Krause:2010a}.  Let $H$ be the composite of the
canonical functor $\sfT\to\sfT/\sfS$ with the restricted Yoneda
functor $\sfT/\sfS\to \Abel_\beta(\sfT/\sfS)$. Then $\Ker H=\sfS$.

(4) $\Rightarrow$ (5): The Verdier quotient $\sfT/\sfS$ has the
property that $\Hom_{\sfT/\sfS}(X,Y)$ is a set for each pair of
objects $X,Y$ in $\sfT/\sfS$; see \cite[Corollary 4.4.3]{Neeman:2001a}
or \cite[Theorem~7.2.1]{Krause:2010a}. Thus the canonical functor
$Q\col \sfT\to\sfT/\sfS$ admits a right adjoint, by Brown
representability~\cite[Theorem 8.3.3]{Neeman:2001a}.  The composite of $Q$ with its right adjoint is a
localization functor with kernel $\sfS$.
\end{proof}

\begin{remark} The triangulated category $\sfS$ is well generated if
  and only if (3), equivalently (4), holds.
\end{remark}

A \emph{skeleton} of a category $\sfC$ is a full subcategory $\sfS$
such that each object in $\sfC$ is isomorphic to exactly one object in
$\sfS$; this is unique up to an isomorphism of categories. We write $|\sfC|$ for
the cardinality of the collection of morphism of a skeleton of $\sfC$.

\begin{theorem}
\label{thm:set} Let $\sfT$ be an $\alpha$-well generated triangulated
category and consider the collection $\mcH$ of functors $H\col
\sfT\to\sfA$ such that
\begin{enumerate}
\item $\sfA$ is abelian and has set-indexed coproducts and exact
$\alpha$-filtered colimits,
\item $H$ is cohomological and preserves set-indexed coproducts.
\end{enumerate} Then the localizing subcategories of the form $\Ker H$
for some $H\in\mcH$ form a set of cardinality at most
$2^{2^{|\sfT^\alpha|}}$.
\end{theorem}

\begin{proof} Any functor $H\in\mcH$ can be extended to an exact
  functor $\bar H\col \Abel_\alpha(\sfT)\to\sfA$ preserving
  set-indexed coproducts such that $H=\bar H H_\alpha$, by
  \cite[Proposition~6.10.1]{Krause:2010a}. Thus the kernel of $\bar H$
  is a localizing subcategory of $\Abel_\alpha(\sfT)$ and it
  determines $\Ker H$, since $H_\alpha(X)\neq 0$ for each non-zero
  object $X$ in $\sfT$. Recall that a full subcategory of an abelian
  category is \emph{localizing} if it is closed under subobjects,
  quotient objects, extensions, and set-indexed coproducts.

  The abelian category $\Abel_\alpha(\sfT)$ is generated by the
  representable functors $H_\alpha(X)$ with $X\in\sfT^\alpha$, and it
  follows that each object of $\Abel_\alpha(\sfT)$ is the
  $\alpha$-filtered union of its subobjects of the form
  $H_\alpha(X)/U$ where $X\in\sfT^\alpha$ and $U\subseteq H_\alpha(X)$
  is a subobject.  Thus $\Ker\bar H$ is determined by the objects of
  the form $H_\alpha(X)/U$ which it contains. The number of pairs
  $(U,X)$ where $X\in\sfT^\alpha$ and $U\subseteq H_\alpha(X)$ is
  bounded by ${2^{|\sfT^\alpha|}}$, since $U$ is given by the family
  of subsets $U(C)\subseteq \Hom_\sfT(C,X)$ where $C$ runs through the
  objects of $\sfT^\alpha$.  It follows that the number of
  subcategories of the form $\Ker \bar H$ is bounded by the cardinal
  $2^{2^{|\sfT^\alpha|}}$.
\end{proof}

For the preceding result, the most interesting case is when
$\alpha=\aleph_0$, which means that $\sfT$ is compactly generated. In
that case $\Abel_\alpha(\sfT)$ equals the category of
$\sfT^c$-modules, where the category of compact objects $\sfT^c$ is
viewed as a ring with several objects. Given a localizing subcategory
$\sfC\subseteq\Abel_\alpha(\sfT)$, there exists a set of
indecomposable injective modules $Q_i$ such that a module $X$ belongs
to $\sfC$ if and only if $\Hom_{\sfT^c}(X,Q_i)=0$ for all $i$; see
\cite[Chap.~III]{Gabriel:1962a}. The isomorphism classes of
indecomposable injective $\sfT^c$-modules form a set, which can be
identified (via Brown representability) with the \emph{Ziegler
  spectrum} $\Sp_{\mathrm{Zg}}(\sfT)$ consisting of the isomorphism
classes of indecomposable pure-injective objects in $\sfT$; see
\cite[\S1]{Krause:2000}. The map taking a localizing subcategory
$\sfS\subseteq\sfT$ as in Theorem~\ref{thm:set} to $\sfS^\perp\cap
\Sp_{\mathrm{Zg}}(\sfT)$ is injective, since $\sfS={^\perp
  (\sfS^\perp\cap \Sp_{\mathrm{Zg}}(\sfT))}$. This provides another
explicit method to bound the cardinality of the collection of these
localizing subcategories.

\section{Bousfield classes}
\label{sec:Bousfield classes}

In this section we recall the notion of a Bousfield class of an object
in a tensor triangulated category, and prove that the collection of
Bousfield classes form a set, provided the triangulated category is
well generated.

\subsection*{Tensor triangulated categories}
Let $(\sfT,\otimes,\one)$ be a \emph{tensor triangulated
  category}. Thus $\sfT$ is a triangulated category with a symmetric
monoidal structure; $\otimes$ is its tensor product which is exact in
each variable and $\one$ is the unit of the tensor product. We assume
that $\sfT$ has set-indexed coproducts and that the tensor product
preserves coproducts in each variable.

Recall that a subcategory $\sfS\subseteq\sfT$ is \emph{tensor closed}
if $X\in\sfS$ and $Y\in\sfT$ implies $X\otimes Y\in\sfS$. Given any
object or class of objects $X$ in $\sfT$, we write $\tloc(X)$ for the
smallest tensor closed localizing subcategory of $\sfT$ containing
$X$.

For a subcategory $\sfS\subseteq\sfT$, we define its \emph{orthogonal subcategories}
\begin{align*} 
\sfS^\perp&=\{Y\in\sfT\mid\Hom_\sfT(X,Y)=0\text{ for all }X\in\sfS\},\\ 
^\perp\sfS&=\{X\in\sfT\mid\Hom_\sfT(X,Y)=0\text{ for all }Y\in\sfS\}.
\end{align*}
It is not hard to verify that $\sfS\subseteq{^\perp(\sfS^\perp)}$, and that equality holds when
$\sfS=\Ker L$ for some localization functor $L\col\sfT\to\sfT$.

\subsection*{Bousfield classes} 
Let $X$ be an object in $\sfT$. Following Bousfield
\cite{Bousfield:1979a}, we define the \emph{Bousfield class} of an
object $X$ to
be the full subcategory of $\sfT$ with objects
\[ 
\acy(X) = \{Y\in \sfT\mid X\otimes Y=0\}\,.
\] 
The objects in $\acy(X)$ are said to be \emph{$X$-acyclic}. Note that
$\acy(X)$ is a tensor closed localizing subcategory of $\sfT$.

For the stable homotopy category of spectra, Ohkawa proved
\cite{Ohkawa:1989a} that the collection of Bousfield classes forms a
set of cardinality at most $2^{2^{\aleph_0}}$; see also
\cite{Dwyer/Palmieri:2001a}. More recently, Dwyer and
Palmieri~\cite{Dwyer/Palmieri:2008a} investigated the Bousfield
classes in the derived category of modules over some non-noetherian
rings, and asked in Question~5.9 of op.~cit., whether these Bousfield
classes form a set. We answer their question in the affirmative, as
follows\footnote{See also \cite{Casacuberta/Gutierrez/Rosicky:2012a} for other results in this direction.}.

\begin{theorem} 
\label{th:cardinality}
For any $\alpha$-well generated tensor triangulated category $\sfT$,
the collection of Bousfield classes forms a set of cardinality at most
$2^{2^{|\sfT^\alpha|}}$.
\end{theorem}

\begin{proof} 
  The Bousfield class of an object $X$ is the kernel of the functor
  $X\otimes -$, which is exact and preserves coproducts. It thus
  follows from Proposition~\ref{prop:loc} and Theorem~\ref{thm:set}
  that the collection of such localizing subcategories forms a set of
  cardinality at most $2^{2^{|\sfT^\alpha|}}$.
\end{proof}

\subsection*{The Bousfield lattice}

A \emph{lattice} is by definition a partially ordered set $\Lambda$
with the property that for each pair of elements $a,b$ in $\Lambda$
there is a supremum, denoted $a\vee b$, and an infimum, denoted
$a\wedge b$. A lattice $\Lambda$ is \emph{complete} if for any subset
$A\subseteq\Lambda$ the supremum $\bigvee_{a\in A} a$ and the infimum
$\bigwedge_{a\in A} a$ exist.  In any partially ordered set the
infimum can be expressed as a supremum and vice versa. For instance,
\[ 
\bigwedge_{a\in A} a= \bigvee_{b\in B}b \qquad \text{where } B=\{b\in\Lambda\mid b\leq a \text{ for all }a\in A\}\,.
\]
Thus, $\Lambda$ is complete if every subset in it has a supremum;
equivalently, if every subset has an infimum. In any lattice, we write
$0$ for the unique minimal element and $1$ for the unique maximal
element, provided they exist.

Let $\sfT$ be a well generated tensor triangulated category and denote
by $\Acy(\sfT)$ the set of Bousfield classes in $\sfT$. As in
\cite{Bousfield:1979a}, there is a partial order $\le$ on $\Acy(\sfT)$
given by:
\[ 
\acy(X)\le \acy(Y)\quad \text{when}\quad \acy(X)\supseteq \acy(Y)\,.
\]
For any set of objects $X_i$ in $\sfT$, one has
\[ 
\bigvee_{i} \acy(X_{i}) =\acy(\coprod_{i}X_{i})\,.
\] 
Thus $\Acy(\sfT)$ is a complete lattice and we call it the \emph{Bousfield lattice} of $\sfT$. 

We consider also the collection of tensor closed localizing
subcategories of the form $\tloc(X)$, and denote it
$\Tloc(\sfT)$. This time the partial order considered is the obvious
one, namely, the one given by inclusion. Without additional hypotheses
we do not know whether $\Tloc(\sfT)$ is a set or a proper class; see
however Corollary~\ref{cor:locset}.

In $\Tloc(\sfT)$ any subset has a supremum, given as before by
\[ 
\bigvee_{i} \tloc(X_{i}) = \tloc(\coprod_{i}X_{i})\,.
\] 
The infimum has the following explicit description.

\begin{lemma} 
Assume $\sfT$ is well generated. Then there is in $\Tloc(\sfT)$ an equality
\[ 
\bigwedge_i \tloc(X_i)=\bigcap_i\tloc(X_i)
\]
for every set of objects $X_i\in\sfT$.
\end{lemma}

\begin{proof} One needs to find an object $X\in\sfT$ such that
$\bigcap_i\tloc(X_i)=\tloc(X)$. It follows from
Proposition~\ref{prop:loc} that there are regular cardinals $\beta_i$
and cohomological functors $H_i\col\sfT\to\sfA_i$ such that $\sfA_i$
has exact $\beta_i$-filtered colimits and $\tloc(X_i)=\Ker H_i$ for
all $i$. Put $\sfA=\prod_i\sfA_i$ and observe that $\sfA$ has
$\beta$-exact filtered colimits for $\beta=\sum_i\beta_i$. Clearly,
$\bigcap_i\tloc(X_i)=\Ker H$ for the functor $H\col\sfT\to\sfA$ taking
an object $Y$ to $(H_iY)$. Thus $\Ker H=\tloc(X)$ for some $X\in\sfT$
by Proposition~\ref{prop:loc}.
\end{proof}

\begin{remark}
  If $\Tloc(\sfT)$ is a set, then every tensor closed localizing
  subcategory $\sfS\subseteq\sfT$ belongs to $\Tloc(\sfT)$, since
  $\sfS=\bigvee_{X\in\sfS}\tloc(X)$.
\end{remark}

\section{Stratification}
\label{se:stratification}

In this section we determine the structure of the Bousfield lattice,
and of the lattice of localizing subcategories, for some classes of
tensor triangulated categories. This involves the stratification
property introduced and studied in
\cite{Benson/Iyengar/Krause:2008b,Benson/Iyengar/Krause:2009a}.

For the remainder of this section and the next, we assume
$(\sfT,\otimes,\one)$ is \emph{compactly generated tensor triangulated
  category}, by which we mean that the following conditions are
satisfied:
\begin{enumerate}
\item $\sfT$ is a compactly generated triangulated category.
\item The  unit $\one$ is compact and all compact objects are strongly dualizable.
\item The functor $\fHom(-,Y)$ is exact for each object $Y$ in $\sfT$.
\end{enumerate}
Here, $\fHom(X,-)$ denotes for each object $X$ in $\sfT$ the right
adjoint of the tensor functor $X\otimes-\col\sfT\to\sfT$, which
exists by Brown representability. Thus
\begin{equation*}
\label{eq:adj} \Hom_\sfT(X \otimes Z,Y) \cong \Hom_\sfT(Z,\fHom(X,Y))
\end{equation*} 
for all objects $Y,Z$ in $\sfT$. We write $X^\vee=\fHom(X,\one)$, and the object $X$ is
\emph{strongly dualizable} if the canonical morphism
\[
X^\vee \otimes Y \to \fHom(X,Y)
\] 
is an isomorphism for all $Y$ in $\sfT$.

\subsection*{Actions} 
Fix a graded-commutative noetherian ring $R$ and a homomorphism
$R\to \End_\sfT^*(\one)$ into the graded endomorphism ring of
$\one$. In this way $R$ acts on $\sfT$, that is, the graded abelian
group
\[ 
\Hom^*_\sfT(X,Y)=\bigoplus_{n\in\bbZ}\Hom_\sfT(X,\Si^nY)
\] 
is an $R$-module, for all objects $X,Y$.

Let $\Spec R$ denote the set of graded prime ideals in $R$. In
\cite[\S5]{\bik:2008a} and \cite[\S4]{\bik:2010a}, we constructed for
each $\fp\in\Spec R$ an adjoint pair of exact functors
$\gam_\fp\col\sfT\to\sfT$ and $\lam^\fp\col\sfT\to\sfT$, and defined
for each object $X$ \emph{support} and \emph{cosupport} as follows:
\begin{align*} \supp_R(X)&=\{\fp\in\Spec R\mid\gam_\fp X\neq 0\}\\
\cosupp_R(X)&=\{\fp\in\Spec R\mid\lam^\fp X\neq 0\}.
\end{align*} Note that $\gam_\fp X\cong \gam_\fp\one\otimes X$ and
$\lam^\fp X\cong \fHom(\gam_\fp\one,X)$.

\begin{lemma}
\label{lem:supp} 
Given objects $X,Y$ in $\sfT$, there are implications:
\[ 
\tloc(X)\subseteq \tloc(Y) \quad\implies\quad \acy(X)\le \acy(Y)
\quad\implies\quad \supp_R(X)\subseteq\supp_R(Y)\,.
\]
\end{lemma}

\begin{proof} 
The first implication follows from the fact that the functor $W\otimes -$ is exact and preserves set-indexed coproducts, for any object $W$ in $\sfT$.

For the second implication, observe that for any object $X$ in $\sfT$,
a prime $\fp$ is in $\supp_R(X)$ if and only if
$\gam_\fp\one\not\in\acy(X)$.
\end{proof}

Next we establish converses to the preceding lemma. We begin with a reformulation of some results from \cite{\bik:2008b,\bik:2010a}. Recall that
\[ 
\supp_R(\sfT)=\{\fp\in\Spec R\mid \text{$\fp\in\supp_R(X)$ for some $X\in\sfT$}\}\,.
\] 
This set coincides with $\supp_R(\one)$, since $\gam_{\fp}X\cong \gam_\fp\one\otimes X$.

\subsection*{Function objects}
The result below provides a formula for the cosupport of function
objects. In the language of \cite{\bik:2009a}, condition (1) is the
statement that $\sfT$ is \emph{stratified} by $R$, as a tensor
triangulated category.

\begin{theorem}
\label{thm:comparison} 
The following conditions are equivalent:
\begin{enumerate}
\item $\tloc(\gam_{\fp}\one)$ is a minimal non-zero element of
$\Tloc(\sfT)$, for each $\fp\in\supp_R(\sfT)$.
\item $\cosupp_R(\fHom(X,Y))=\supp_R(X)\cap \cosupp_R(Y)$ for all
$X,Y$ in $\sfT$.
\item $\tloc(X)=\{Y\in\sfT\mid \supp_R(Y)\subseteq\supp_R(X)\}$ for all
$X$ in $\sfT$.
\item $\tloc(X)\subseteq\tloc(Y) \iff \supp_R(X)\subseteq\supp_R(Y)$, for all $X,Y$ in
$\sfT$.
\end{enumerate}
\end{theorem}

\begin{proof} The equivalence of (1) and (2) is
\cite[Theorem~9.5]{\bik:2010a}, while that of (1), (3) and (4) follows
from \cite[Theorem~3.8]{\bik:2008b}.
\end{proof}

\begin{corollary}
\label{cor:locset} 
Suppose that $\sfT$ is stratified as a tensor triangulated category by
a graded-commutative noetherian ring $R$. Then the lattice of tensor
closed localizing subcategories of $\sfT$ is isomorphic to the
lattice of subsets of $\supp_R(\sfT)$ via the map sending $\tloc(X)$
to $\supp_R(X)$.\qed
\end{corollary}

Next we formulate the analogue of Theorem~\ref{thm:comparison} for
Bousfield classes. 

\subsection*{Tensor products}
The following result provides a formula for the support of tensor
products. We say that the \emph{tensor product formula} holds in
$\sfT$ when the conditions below are satisfied.

\begin{theorem}
\label{thm:comparison_bclass} 
The following conditions are equivalent:
\begin{enumerate}
\item $\acy(\gam_{\fp}\one)$ is a minimal non-zero element of
$\Acy(\sfT)$, for each $\fp\in\supp_R(\sfT)$.
\item $\supp_{R}(X\otimes Y)=\supp_R(X)\cap \supp_R(Y)$ for all $X,Y$
in $\sfT$.
\item $\acy(X) = \{Y\in\sfT\mid\supp_R(Y)\cap\supp_R(X)=\varnothing\}$
for all $X$ in $\sfT$.
\item $\acy(X)\leq \acy(Y) \iff \supp_R(X)\subseteq \supp_R(Y)$, for all
$X,Y$ in $\sfT$.
\end{enumerate} Moreover, these conditions hold when $\sfT$ is
stratified by $R$.
\end{theorem}

\begin{proof} 
(1) $\Rightarrow$ (2): When (1) holds, for each $X\in\sfT$ and
$\fp\in\supp_R(X)$, one has
\[ \acy(\gam_{\fp}X)=\acy(\gam_{\fp}\one)\,.
\] Indeed, the isomorphism $\gam_{\fp}X\cong \gam_{\fp}\one\otimes X$
implies $\acy(\gam_{\fp} X)\le \acy(\gam_{\fp} \one)$. The
equality holds since $\acy(\gam_{\fp}X)\neq \sfT$.

Now for any $X$ and $Y$ one has $\supp_{R}(X\otimes Y)\subseteq \supp_R(X)\cap\supp_R(Y)$. When $\fp$ is not in
$\supp_{R}(X\otimes Y)$, one gets
\[ 
\gam_{\fp}(X) \otimes Y\cong \gam_{\fp}(X\otimes Y) =0\,,
\] 
so that $Y\in \acy(\gam_{\fp}X)$; if in addition $\fp$ is in $\supp_{R}(X)$, then since $\acy(\gam_{\fp}X)=\acy(\gam_\fp\one)$, it follows that $\fp\not\in\supp_R(Y)$.  Thus formula in (2) holds.

(2) $\Rightarrow$ (3): This is immediate, for $X\otimes Y=0$ if and
only if $\supp_{R}(X\otimes Y)=\varnothing$, by
\cite[Theorem~5.2]{\bik:2008a}.

(3) $\Rightarrow$ (4): It is clear that when (3) holds and $\supp_R(X)\subseteq \supp_R(Y)$, one has $\acy(X)\leq
\acy(Y)$. The reverse implication holds always; see Lemma~\ref{lem:supp}.

(4) $\Rightarrow$ (1): This is clear, for $\supp_{R}(\gam_{\fp}\one)=\{\fp\}$ for $\fp\in\supp_{R}(\sfT)$.

For the last conclusion, observe that part (4) of Theorem~\ref{thm:comparison} in combination with
Lemma~\ref{lem:supp} implies part (4) of the present theorem.
\end{proof}

\begin{corollary}\label{cor:bousfield-lattice} 
  Suppose that the tensor product formula holds in $\sfT$, for
  instance, when $\sfT$ is stratified by $R$ as a tensor triangulated
  category. Then the Bousfield lattice of $\sfT$ is isomorphic to the
  lattice of subsets of $\supp_R(\sfT)$ via the map sending $\acy(X)$
  to $\supp_R(X)$. In particular, for all $X,Y$ in $\sfT$,
\[ 
  \acy(X)\wedge\acy(Y) = \acy(X\otimes Y).
\]
\end{corollary}

\begin{proof}
  The inverse map sends $\mcU\subseteq\supp_R(\sfT)$ to
  $\acy(\coprod_{\fp\in\mcU}\gam_\fp\one)$.  Indeed, note that
\[
\tloc(X)=\bigvee_{\fp}\tloc(\gam_\fp X),
\]
by \cite[Theorem~7.2]{\bik:2009a}, where $\fp$ runs through all primes in $\Spec R$. This implies
\[
\acy(X)=\bigvee_{\fp\in\supp_R(X)}\acy(\gam_\fp X)=\bigvee_{\fp\in\supp_R(X)}\acy(\gam_\fp \one),
\] 
where the minimality of $\acy(\gam_\fp \one)$ is used for the second equality. It follows that the maps between $\Acy(\sfT)$ and subsets of $\supp_R(\sfT)$ are mutually inverse.
\end{proof}

\begin{remark} 
  In \cite{Hovey/Palmieri/Strickland:1997a}, condition (1) of Theorem~\ref{thm:comparison} is formulated as Conjecture~6.1.2, while condition (1) of Theorem~\ref{thm:comparison_bclass} is formulated as Conjecture~6.1.3 (with some additional assumptions on $\sfT$).  A number of interesting consequences are proved in \cite[Theorem~6.1.5]{Hovey/Palmieri/Strickland:1997a}, including a classification of the thick subcategories of $\sfT^{\sfc}$.
\end{remark}

\begin{example}
\label{ex:ca} Let $A$ be a commutative noetherian ring. The derived
category of the category of $A$-modules is a compactly generated
tensor triangulated category with a natural $A$-action. It follows
from Neeman's work \cite{Neeman:1992a}, see also
\cite[Theorem~8.1]{\bik:2009a}, that this category is stratified by
$A$, that is to say, the equivalent conditions in
Theorem~\ref{thm:comparison} hold.
\end{example}

\begin{remark} In the context of Example~\ref{ex:ca}, Dwyer and
Palmieri~\cite[p.~429]{Dwyer/Palmieri:2008a} ask if
$\acy(X)=\acy(\bigoplus_nH^nX)$ for any complex $X$ in the derived
category of $A$. However, it is clear from
Theorem~\ref{thm:comparison_bclass} that this cannot hold in general,
for there exist complexes $X$ such $\supp_R(X)\ne
\supp_{R}(\bigoplus_nH^nX)$; see \cite[Example~9.4]{\bik:2008a}.
\end{remark}

\begin{example} 
  Let $k$ be a field and consider the category of $k$-linear maps
  $V\to W$. This is an abelian category, and the tensor product over $k$
  induces on it a symmetric monoidal structure with tensor identity
  $\one=(k\xra{\id} k)$. The corresponding derived category $\sfD$ is
  then a compactly generated tensor triangulated category. Set
  $R=\End_\sfD^*(\one)=k$. Then the tensor product formula does not
  hold in $\sfD$.

Take for instance $X=(k\to 0)$ and $Y=(0\to k)$. Then $X\otimes Y=0$, and so
\[ 
\supp_R(X\otimes Y)=\varnothing\neq\Spec R=\supp_R(X)\cap\supp_R(Y)\,.
\]
\end{example}

\section{Local objects}
\label{sec:Local objects}

Let $(\sfT,\otimes,\one)$ be a compactly generated tensor triangulated
category, as in Section~\ref{se:stratification}. We follow
\cite{Bousfield:1979a} and say for a given object $X$ that an object
$Y$ is \emph{$X$-local} if the functor $\Hom_{\sfT}(-,Y)$ annihilates
all $X$-acyclic objects. An object is \emph{$X$-acyclic} if it is
annihilated by $X\otimes -$.  In this section we interpret this
property in terms of support and cosupport of $X$ and $Y$,
respectively, when these notions are defined.

We begin with an elementary observation; recall that $X^\vee=\fHom(X,\one)$.

\begin{proposition}
\label{prop:compact-local}
When $X\in \sfT$ is compact, $\acy(X)=\acy(X^{\vee})$ and $X$ is $X$-local.
\end{proposition}

\begin{proof}
The object $X$ is strongly dualizable and therefore a retract of $X\otimes
  X^{\vee}\otimes X$; see \cite[Proposition~III.1.2]{Lewis:1986a}.  It
  follows that
\[
\acy(X)\leq \acy(X\otimes X^\vee\otimes X)\leq \acy(X^{\vee})\,.
\]
Since $(X^{\vee})^{\vee}\cong X$, the inequalities above yield also
$\acy(X^{\vee})\leq \acy(X)$.

Now when $Y$ is $X$-acyclic, it is also $X^{\vee}$-acyclic, which
yields the equality below:
\[
\Hom_{\sfT}(Y,X)\cong \Hom_{\sfT}(Y,(X^{\vee})^{\vee})\cong\Hom_{\sfT}(X^{\vee}\otimes Y,\one)=0
\]
The first isomorphism holds as $(X^{\vee})^{\vee}\cong X$, and the
second one is by adjunction. Therefore $X$ is $X$-local, as claimed.
\end{proof}

The preceding result does not, in general, extend to non-compact
objects; see Example~\ref{ex:ca2}. A basic problem is to recognize
when an object is $X$-local, and this has a satisfactory answer if
$\sfT$ is stratified; this is explained next.

\subsection*{Cosupport}
Henceforth we assume that $\sfT$ is endowed with an action of a
graded-commutative noetherian ring $R$, as in
Section~\ref{se:stratification}. Given a subset
$\mcU\subseteq\Spec R$, consider full subcategories
\[ 
\sfT_\mcU =\{X\in\sfT\mid\supp_R(X)\subseteq\mcU\}\quad\text{and}\quad \sfT^\mcU = \{X\in\sfT\mid\cosupp_R(X)\subseteq\mcU\}\,.
\]
These subcategories are related to each other:

\begin{lemma}
\label{le:perp} 
Let $\mcU\subseteq\Spec R$ and $\mcU'=\Spec R\setminus\mcU$. Then
\[ 
(\sfT_\mcU)^\perp=\sfT^{\mcU '} \quad\text{and}\quad \sfT_\mcU={^\perp(\sfT^{\mcU '})}\,.
\]
\end{lemma}

\begin{proof} 
Fix an object $Y$ in $\sfT$. By definition, $\cosupp_R(Y)\subseteq \mcU'$ if and only if
\[ 
\Hom_\sfT(\gam_\fp-,Y)\cong\Hom_\sfT(-,\lam^\fp Y)=0
\] 
for all $\fp\in\mcU$. The local-global principle, \cite[Theorem~3.6]{\bik:2008a} implies
\[ 
\sfT_\mcU=\tloc(\{\gam_\fp X\mid X\in\sfT,\,\fp\in\mcU\})\,.
\] 
Thus $(\sfT_\mcU)^\perp=\sfT^{\mcU '}$.

On the other hand, $\sfT_\mcU=\bigcap_{\fp\in\mcU'}\Ker \gam_\fp$. It
follows from Proposition~\ref{prop:loc} that $\sfT_\mcU$ is the kernel
of a localization functor. Thus
$\sfT_\mcU={^\perp((\sfT_\mcU)^\perp)}={^\perp(\sfT^{\mcU '})}$.
\end{proof}

\begin{proposition} 
\label{prop:tp=ac}
The tensor product formula holds in $\sfT$ if and only if
\[ 
\acy(X)^\perp =
\{Y\in\sfT\mid\cosupp_R(Y)\subseteq\supp_R(X)\}\qquad\text{for all }X\in\sfT\,.
\]
\end{proposition}

\begin{proof} 
  Given the description of the Bousfield class $\acy(X)$ in
  Theorem~\ref{thm:comparison_bclass}, this is an immediate
  consequence of Lemma~\ref{le:perp}.
\end{proof}

One consequence is that the Dichotomy Conjecture of Hovey and Palmieri
\cite[Conjecture~7.5]{Hovey/Palmieri:1999a} holds in our setting. The
conjecture asserts that for each object $X$ in $\sfT$, there exists a
non-zero compact object that is either $X$-acyclic or $X$-local.

\begin{example}
  Suppose the tensor product formula holds in $\sfT$.

  Let $\fp\in\supp_R(\sfT)$ be maximal and denote by $\kos{\one}{\fp}$
  the corresponding Koszul object \cite[\S5]{\bik:2008a}. By its
  construction, the object is compact, and
\[
\supp_R( \kos{\one}{\fp})=\{\fp\}=\cosupp_R(\kos{\one}{\fp}),
\]
by \cite[Lemma~2.6]{\bik:2009a} and \cite[Lemma~4.12]{\bik:2010a}.
Given any object $X$ in $\sfT$, there are two possible cases. If $\fp$
is in $\supp_{R}(X)$, then the object $\kos{\one}{\fp}$ is $X$-local, by
Proposition~\ref{prop:tp=ac}. Otherwise, $\kos{\one}{\fp}$ is
$X$-acyclic, by Theorem~\ref{thm:comparison_bclass}.
\end{example}

The following example is intended to show that not ever object is local with respect to itself; confer Proposition~\ref{prop:compact-local}.

\begin{example}
  \label{ex:ca2} 
Let $A$ be a commutative noetherian ring. Fix a prime ideal $\fp$ and let $E$ be the injective hull of $A/\fp$. The
derived category of $A$ is then a noetherian $R$-linear category, where $R=A$, and there are equalities
\[ 
\cosupp_R(E) = \{\fq\in\Spec R\mid\fq\subseteq\fp\}\quad\text{and}\quad \supp_R(E) = \{\fp\}.
\] 
For the computation of $\cosupp_R(E)$, see \cite[Proposition~5.4]{\bik:2010a}. Thus, $E$ is not $E$-local.
Recall that the derived category is stratified by $A$.
\end{example}

The preceding discussion raises the question: When is $\cosupp_R(X)\subseteq \supp_R(X)$? 

\begin{proposition}
\label{prop:cosupp-supp}
In any $R$-linear tensor triangulated category $\sfT$, each compact
object $X$ satisfies $\cosupp_{R}(X)\subseteq \supp_{R}(X)$.
\end{proposition}

\begin{proof}
Suppose $\fp$ is not in $\supp_{R}(X)$, and fix a compact object
$C$. Then one has
\[
X\otimes \gam_{\fp} C\cong X\otimes \gam_{\fp}\one \otimes C \cong \gam_{\fp}X\otimes C =0\,.
\]
Therefore $\gam_{\fp} C$ is $X$-acyclic. Using that $X$ is $X$-local,
by Proposition~\ref{prop:compact-local}, this yields
\[
\Hom_{\sfT}(C,\lam^{\fp}X)\cong \Hom_{\sfT}(\gam_{\fp}C,X)=0\,.
\]
Thus $\lam^{\fp}X=0$, since $C$ was arbitrary. It follows that $\fp$ is not in $\cosupp_{R}(X)$.
\end{proof}

The next example shows that the inclusion in the preceding result can be strict.

\begin{example}
\label{ex:kg} Let $G$ be a finite group, $k$ a field whose characteristic divides the order of $G$, and set $R=H^{*}(G,k)$, the cohomology algebra of $G$.

Then $\KInj{kG}$, the homotopy category of complexes of injective $kG$ modules is a compactly generated triangulated category, admitting a natural $R$-action. As explained in \cite[Example~11.1]{\bik:2010a}, any non-zero compact object $X$ in $\KInj{kG}$ satisfies
\[ 
\cosupp_R(X) =\{H^{\ges1}(G,k)\}\,.
\] 
On the other hand, for any any closed subset $\mcV\subseteq \Spec R$, there exists a compact object $X$ with $\supp_R(X)=\mcV$, so the inclusion in Theorem~\ref{thm:cosupp-supp} can be strict.

In contrast, in $\StMod{kG}$, which is also an $R$-linear compactly
generated triangulated category, for any compact object (that is to
say, for any finite dimensional $kG$-module) $X$ one has $\cosupp_R(X)
= \supp_R(X)$; see \cite[Example~11.14]{\bik:2010a}.
\end{example}

\subsection*{A variation}
To round off this discussion we prove a version of
Proposition~\ref{prop:cosupp-supp} for triangulated categories without
using any tensor structure. Note that the definitions of support and
cosupport in terms of the functors $\gam_\fp$ and $\lam^\fp$ do not
require a tensor structure; see \cite[\S5]{\bik:2008a} and
\cite[\S4]{\bik:2010a} for details.

\begin{theorem}
\label{thm:cosupp-supp} 
Let $\sfT$ be a compactly generated $R$-linear triangulated category and $X$  a compact object such that $\End_\sfT^*(X)$ is finitely generated over $R$. Then
\[ 
\cosupp_R(X)\subseteq \supp_R(X)\,.
\]
\end{theorem}

\begin{proof} 
Fix a prime $\fp\not\in\supp_R(X)$ and a compact object $C$ in $\sfT$. It suffices to prove that $\Hom^{*}_{\sfT}(\gam_{\fp}C,X)=0$. Then adjunction yields $\Hom^{*}_{\sfT}(C,\lam^\fp X)=0$, and hence, since $C$ was arbitrary, $\lam^{\fp}X=0$, that is to say, $\fp\not\in\cosupp_R(X)$.

Since $X$ is compact and the $R$-module $\End^*_\sfT(X)$ is finitely generated, \cite[Theorem~5.5]{\bik:2008a}
yields that $\End^*_\sfT(X)_\fp=0$. Since the $R$-action on $\Hom^{*}_{\sfT}(\gam_{\fp}C, X)$ factors through $\End^*_\sfT(X)$, this then yields the second isomorphism below:
\[
\Hom^{*}_{\sfT}(\gam_{\fp}C, X) \cong \Hom^{*}_{\sfT}(\gam_{\fp}C, X)_{\fp}\cong 0\,.
\]
The first one holds as the $R$-module on its left is $\fp$-local, by \cite[Proposition~2.3]{\bik:2009a}.
\end{proof}

When the $R$-linear category $\sfT$ is noetherian, in the sense of \cite{Benson/Iyengar/Krause:2009a}, 
 the result above may be reformulated as follows; confer also Lemma~\ref{le:perp}.

\begin{corollary} 
  Assume that for any compact object $C$ in $\sfT$ the $R$-module  $\End^{*}_{\sfT}(C)$ is finitely generated. Let $\mcU\subseteq\Spec R$ be a subset such that $\Spec R\setminus\mcU$ is specialization closed. Then there is an equality $\sfT_{\mcU}=\sfT^{\mcU}$, the subcategory 
$\sfS=\sfT_{\mcU}$ is  localizing and colocalizing, and satisfies
\[ 
\sfT^c\cap{^\perp\sfS}\subseteq\sfS^\perp\,.
\]
\end{corollary}

\begin{proof} 
The equality $\sfT_{\mcU}=\sfT^{\mcU}$ is \cite[Corollary~4.9]{\bik:2010a}; from this it follows that $\sfS$ is localizing and also colocalizing.

Setting $\mcV=\Spec R\setminus\mcU$, one gets $^\perp\sfS=\sfT_\mcV$ and $\sfS^\perp=\sfT^\mcV$, by Lemma~\ref{le:perp}. Now the desired inclusion follows from Theorem~\ref{thm:cosupp-supp}.
\end{proof}

\section{Stone duality and support}
\label{se:Stone}
A basic idea in lattice theory, going back to Stone
\cite{Stone:1937a}, is to \emph{represent} a distributive lattice via
its spectrum of prime ideals; see \cite{Borceux:1994a} for a modern
treatment. In this section we consider a well generated tensor
triangulated category $\sfT$ and identify a distributive sublattice of
the Bousfield lattice of $\sfT$. This enables us to introduce a notion
of support for objects in $\sfT$, with values in the associated
topological space; it is an intrinsic notion of
support that does not depend on an action of any ring.

\subsection*{Frames} 
A \emph{frame} is a complete lattice in which the following
\emph{infinite distributivity} holds: for every element $a$ and set of
elements $\{b_i\}$ in $\Lambda$, there is an equality
\[
a\wedge(\bigvee_ib_i)=\bigvee_i(a\wedge b_i)\,.
\] 
An element $p\neq 1$ in a frame $\Lambda$ is called \emph{prime} if
$a\wedge b\le p$ implies $a\le p$ or $b\le p$. We write $\Sp
(\Lambda)$ for the prime elements in $\Lambda$, and for each
$a\in\Lambda$ set
\[
U(a)=\{p\in\Sp(\Lambda)\mid a\not\le p\}\,.
\] 
It is not hard to verify that declaring sets of the form $U(a)$ to be
open defines a topology on $\Sp(\Lambda)$; this is the \emph{Stone
  topology}, and $\Sp(\Lambda)$ with this topology is called the
\emph{spectrum} of $\Lambda$. One says that $\Lambda$ has \emph{enough
  points} if $U(a)=U(b)$ implies $a=b$, for all $a,b\in\Lambda$. This
means that the map sending $a\in \Lambda$ to $U(a)$ induces an
isomorphism between $\Lambda$ and the lattice of open subsets of
$\Sp(\Lambda)$.

A morphism of frames $f\col\Lambda\to\Gamma$ is a map such that for
elements $a,b$ and any set of elements $\{a_i\}$ in $\Lambda$, there are
equalities
\[
f(a\wedge b)=f(a)\wedge f(b)\quad\text{and}\quad f(\bigvee_ia_i)=\bigvee_i f(a_i)\,.
\]
For example, primes in $\Lambda$ correspond to \emph{points}, that is, morphisms $\Lambda\to \mathbb F_2= \{0,1\}$. A morphism of frames $f\col\Lambda\to\Gamma$ induces a continuous map $\Sp(f)\col\Sp(\Gamma)\to\Sp(\Lambda)$ by taking a prime $p$ to $\bigvee_{f(a)\le p}a$. Alternatively, $\Sp(f)$ takes a point
$q\col\Gamma\to\mathbb F_2$ to the composite $qf$. Note that
\begin{equation}
\label{eq:open}
U(f(a))=\Sp(f)^{-1}(U(a))\,.
\end{equation}
This yields a contravariant functor $\Sp$ into the category of topological spaces. Taking a space $X$ to the lattice $\mathcal O(X)$ of open subsets provides a right adjoint to $\Sp$. Thus there is a bijection
\begin{equation}
\label{eq:duality}
\Hom_{\mathsf{Top}}(X,\Sp(\Lambda))\xra{\sim}  \Hom_{\mathsf{Frm}}(\Lambda,\mathcal O(X))
\end{equation}
which sends $f$ to the map $(a\mapsto f^{-1}U(a))$.  For a frame $\Lambda$, the adjunction morphism 
\begin{equation}
\label{eq:adjoint}
\Lambda\lto \mathcal O (\Sp(\Lambda)),\quad a\mapsto U(a)\,,
\end{equation} 
is an isomorphism if and only if $\Lambda$ has enough points. Given a
space $X$, the adjunction morphism $X\to\Sp(\mathcal O(X))$ is a
homeomorphism if and only if $X$ is \emph{sober}, that is, each
non-empty irreducible closed subset has a unique generic point.

The following result summarizes the correspondence between lattices
and spaces; it  goes back to Stone \cite{Stone:1937a} and is known as
\emph{Stone duality}.

\begin{proposition}
\label{pr:frames}
  \pushQED{\qed} Taking a frame to its spectrum and a topological
  space to its lattice of open subsets induces mutually inverse
  (contravariant) equivalences:
\begin{equation*}
\left\{
\begin{gathered}
  \text{frames with enough points}
\end{gathered}\;
\right\}
\xymatrix@C=3pc{ \ar@{<->}[r]^-{\scriptscriptstyle{1-1}} &} \left\{
\begin{gathered}
  \text{sober topological spaces}
\end{gathered}\;
\right\}\qedhere
\end{equation*}
\end{proposition}

\subsection*{Bousfield idempotents}
Let now $\sfT$ be a well generated tensor triangulated category. An
object $X$ in $\sfT$ is said to be \emph{Bousfield idempotent} if
$\acy(X)=\acy(X\otimes X)$. This property only depends on $\acy(X)$,
since $\acy(X)=\acy(Y)$ implies
\[
\acy(X\otimes X)=\acy(X\otimes Y) =\acy(Y\otimes Y)\,.
\]  
Consider the set
\[
\Dl(\sfT)=\{\acy(X)\in\Acy(\sfT)\mid \acy(X)=\acy(X\otimes X)\}
\] 
with the partial order induced from $\Acy(\sfT)$.  The following result is due to Bousfield \cite{Bousfield:1979a}; see also \cite{Hovey/Palmieri:1999a}. It shows that distributivity of the tensor product implies distributivity of $\Dl(\sfT)$; hence our notation.

\begin{proposition}
\label{pr:D(T)}
The partially ordered set $\Dl(\sfT)$ is a frame. More precisely,
\[
\acy(X)\wedge\acy(Y)=\acy(X\otimes Y)\quad\text{and} \quad \bigvee_i\acy(X_i)=\acy(\coprod_iX_i)
\] 
hold in $\Dl(\sfT)$ for all Bousfield idempotent objects $X,Y$ and $\{X_i\}$ in $\sfT$.
\end{proposition}

\begin{proof}
If $U$ is a Bousfield idempotent object with $\acy(U)\le \acy(X)\wedge \acy(Y)$, then
\[
\acy(U)=\acy(U\otimes U)\le \acy(U\otimes Y)\le \acy(X\otimes Y)\,,
\]
where the inequalities are easily verified. If $X$ and $Y$ are  also Bousfield idempotent, then so is $X\otimes Y$, and hence one obtains that $\acy(X)\wedge\acy(Y)=\acy(X\otimes Y)$.

Given a set of Bousfield idempotent objects $X_i$, we have in $\Acy(\sfT)$
\begin{align*}
 \acy(\coprod_{i}X_i)&\ge  \acy\big((\coprod_i X_i)\otimes (\coprod_i
  X_i)\big)=\acy\big(\coprod_{i,j}(X_i\otimes X_j)\big) =
  \bigvee_{i,j}\acy(X_i\otimes X_j) \\ &\ge \bigvee_{i}\acy(X_i\otimes
  X_i) = \bigvee_{i}\acy(X_i) = \acy(\coprod_{i}X_i).
\end{align*}
Thus $\coprod_i X_i$ is Bousfield idempotent, and this implies
$\bigvee_i\acy(X_i)=\acy(\coprod_iX_i)$.

The infinite distributivity in $\Dl(\sfT)$ follows from the fact that
in $\sfT$ the tensor product distributes over set-indexed coproducts.
\end{proof}

The distributive lattice $\Dl(\sfT)$ provides the basis for an intrinsic notion of support.

\subsection*{Support}
We set $\Sp(\sfT)=\Sp(\Dl(\sfT))$ and define for each object $X$ in $\sfT$ its \emph{support} 
\[
\supp_\sfT(X)=\{\acy(P)\in\Sp(\sfT)\mid \acy(X)\not\le \acy(P)\}\,.
\]
By definition, this set is open when $X$ is Bousfield idempotent.

\begin{proposition}
\label{pr:support}
The map $\supp_\sfT(-)$ has the following properties:
\begin{enumerate}
\item $\supp_\sfT(0)=\varnothing$ and $\supp_\sfT(\one)=\Sp(\sfT)$.
\item $\supp_\sfT(\coprod_iX_i)=\bigcup_i\supp_\sfT(X_i)$ for every set of objects $\{X_i\}$ in $\sfT$.
\item $\supp_\sfT(\Si X)=\supp_\sfT(X)$ for every object $X$ in $\sfT$.
\item For every exact triangle $X'\to X\to X''\to $ in $\sfT$ one has
\[
\supp_\sfT(X)\subseteq\supp_\sfT(X')\cup\supp_\sfT(X'')\,.
\]
\item $\supp_\sfT(X\otimes Y)\subseteq\supp_\sfT(X)\cap\supp_\sfT(Y)$ for all objects $X,Y$ in $\sfT$; equality holds when  $X,Y$ are Bousfield idempotent.
\end{enumerate}
\end{proposition}

\begin{proof}
Properties (1)--(4) and the inclusion in (5) hold as the map taking an object $X$ to $\acy(X)$ has the
  analogous properties and  the map taking $\acy(X)$ to
  $\supp_\sfT(X)$ is order preserving.  The second claim in (5) is
  clear from Proposition~\ref{pr:D(T)}.
\end{proof}

For any objects $X,Y$ in $\sfT$, it follows from definitions that 
\[
\acy(X)\le \acy(Y) \quad\implies\quad \supp_\sfT(X)\subseteq\supp_\sfT(Y)\,.
\] 
The  result below establishes a converse, under additional hypotheses on $\sfT$. This provides a criterion for the existence of a reasonable notion of support.

Recall that $\sfT$ is a well generated tensor triangulated category.
 
\begin{proposition}
\label{pr:comparison_bclass} 
The following conditions are equivalent:
\begin{enumerate}
\item   
$\supp_\sfT(X)\neq\varnothing$ for every object $X\neq 0$, and 
\[
\supp_\sfT(X\otimes Y)=\supp_\sfT(X)\cap \supp_\sfT(Y)\quad\text{for all $X,Y$ in $\sfT$.}
\]
\item
$\supp_\sfT(X)$ is an open subset of $\Sp(\sfT)$ for every object $X$, and
\[
\acy(X)\leq \acy(Y) \iff \supp_\sfT(X)\subseteq \supp_\sfT(Y)\quad\text{for all $X,Y$ in $\sfT$.}
\]
\item Every object of $\sfT$ is Bousfield idempotent and $\Acy(\sfT)$ has enough points.
\end{enumerate} 
\end{proposition}

\begin{proof}
  (1) $\Rightarrow$ (2): The assumption implies that for each object
  $X$ in $\sfT$
  \[\acy(X)=\{Y\in\sfT\mid \supp_\sfT(X)\cap
  \supp_\sfT(Y)=\varnothing\}.\]
From $\supp_\sfT(X\otimes X)=\supp_\sfT(X)$ it then follows that $\acy(X\otimes X)=\acy(X)$. Thus
  $X$ is Bousfield idempotent, and so $\supp_\sfT(X)$ is open.

  If $\acy(X)\not\leq \acy(Y)$, then there exists a $U$ such that
  $U\otimes X\neq 0$ and $U\otimes Y= 0$. Thus $\supp_{\sfT}(U)\cap
  \supp_{\sfT}(X)\neq\varnothing$ while $\supp_{\sfT}(U)\cap
  \supp_{\sfT}(Y)=\varnothing$, and hence $\supp_\sfT(X)\not\subseteq
  \supp_\sfT(Y)$. The other implication always holds.

(2) $\Rightarrow$ (3): The frame $\Dl(\sfT)$ has enough points, since $\supp_\sfT(X)=\supp_\sfT(Y)$ implies $\acy(X)=\acy(Y)$. 

If an object $X$ in $\sfT$ is such that $\supp_\sfT(X)$ is open, then there exists a Bousfield idempotent object $Y$ such that $\supp_\sfT(X)=\supp_\sfT(Y)$; this implies $\acy(X)=\acy(Y)$, and hence $X$ is also Bousfield idempotent.

  (3) $\Rightarrow$ (1): The equality $\supp_\sfT(X\otimes
  Y)=\supp_\sfT(X)\cap \supp_\sfT(Y)$ holds for all Bousfield
  idempotent objects $X,Y$ by Proposition~\ref{pr:support}. If $X\neq
  0$, then $\acy(X)\neq \acy(0)$, and therefore
  $\supp_\sfT(X)\neq\varnothing$ since the Bousfield lattice has
  enough points.
\end{proof}

\begin{remark}
In the stable homotopy category of spectra, there are objects that are not Bousfield idempotent \cite[Lemma~2.5]{Bousfield:1979a}; see \cite[Theorem~6.1]{Dwyer/Palmieri:2008a} for examples in the derived category of a ring. Also, a priori it is possible that $\Sp(\sfT)=\varnothing$.

On the other hand, any triangulated category $\sfT$ which is
stratified by the action of a graded-commutative noetherian ring
satisfies the equivalent conditions of
Proposition~\ref{pr:comparison_bclass}; see Example~\ref{ex:loc}
below.
\end{remark}

\subsection*{Universality}
The map $\supp_\sfT(-)$ enjoys the following universal property. Its
statement is inspired by the universality of Balmer's support
\cite{Balmer:2005a}, and is an immediate consequence of Stone duality.

\begin{proposition}
\label{pr:universal}
Suppose each object in $\sfT$ is Bousfield idempotent. Let $U$ be a
topological space and $\sigma$ a map that assigns to each object
$X\in\sfT$ an open subset $\si(X)$ of $U$ with the following
properties:
\begin{enumerate}
\item $\si(X\otimes Y)=\si(X)\cap\si(Y)$ for all $X,Y$ in $\sfT$,
\item $\si(\coprod_i X_i)=\bigcup_i\si(X_i)$ for every set of objects $X_i$ in $\sfT$,
\item $\acy(X)=\acy(Y)\implies\si(X)=\si(Y)$,  for all $X,Y$ in $\sfT$.
\end{enumerate}
Then there exists a unique continuous map $f\col U\to\Sp(\sfT)$ such
that 
\[
\si(X)=f^{-1}(\supp_\sfT(X))\qquad\text{for all }X\in\sfT\,.
\]
\end{proposition}

\begin{proof}
The map $\si$ yields a frame morphism $\Dl(\sfT)\to\mathcal O(U)$
which then corresponds uniquely to a continuous map $f\col U\to\Sp(\sfT)$, by Stone
duality \eqref{eq:duality}.
\end{proof}

\begin{example}\label{ex:loc}
  Let $\sfT$ be a compactly generated tensor triangulated category
  stratified via the action of a graded-commutative noetherian ring
  $R$. Consider the set $\supp_R(\sfT)$ endowed with the discrete
  topology. It follows from Theorem~\ref{thm:comparison_bclass} that
  the map $\supp_R(-)$ satisfies all the properties listed in
  Proposition~\ref{pr:universal}; it thus induces a continuous map
  $\supp_R(\sfT)\to\Sp(\sfT)$. This map sends $\fp$ in $\supp_R(\sfT)$
  to $\acy(\coprod_{\fq\ne\fp}\gam_\fq\one)$ and is actually a
  homeomorphism, since $\supp_R(-)$ gives an isomorphism
\begin{equation*}
\Dl(\sfT)=\Acy(\sfT)\stackrel{\sim}\lto\mathbf 2^{\supp_R(\sfT)}=\mathcal O(\supp_R(\sfT)),
\end{equation*}
by Corollary~\ref{cor:bousfield-lattice}.
\end{example}

\subsection*{Functoriality}
A functor $F\col\sfT\to\sfU$ between tensor triangulated categories is
called \emph{tensor triangulated} if it is an exact functor that
respects the monoidal structures; we do not assume that $F$ preserves
the tensor unit. We call such a functor \emph{conservative} if
$\acy(X)=\acy(Y)$ implies $\acy(FX)=\acy(FY)$ for all objects $X,Y$ in
$\sfT$.

\begin{proposition}
  Let $F\col\sfT\to\sfU$ be a tensor triangulated functor between well
  generated tensor triangulated categories. Assume $F$ is conservative
  and preserves set-indexed coproducts.  Then the map sending
  $\acy(X)$ to $\acy(FX)$ induces a morphism of frames
  $\Dl(\sfT)\to\Dl(\sfU)$ and hence a continuous map $\Sp(F)\col
  \Sp(\sfU)\to\Sp(\sfT)$. For each object $X$ in $\sfT$, one has
\[
\supp_\sfU(FX)=\Sp(F)^{-1}(\supp_\sfT(X))\,.
\]
\end{proposition}

\begin{proof}
The proof is straightforward.
\end{proof}

We do not have a general criterion for a functor to be conservative,
but there is the following important example.

\begin{proposition}\label{pr:conservative}
  Let $\sfT$ be a well generated tensor triangulated category and
  $\sfS$ a tensor closed localizing subcategory such that $\sfS$ is
  well generated and $\sfS^\perp$ is tensor closed. Consider the
  induced tensor triangulated structure for the Verdier quotient
  $\sfT/\sfS$. Then the inclusion functor $\sfS\to\sfT$, the quotient
  functor $\sfT\to\sfT/\sfS$, and their right adjoints are tensor
  triangulated and conservative. These functors induce isomorphisms
  \[\Acy(\sfT)\xra{\sim}\Acy(\sfS)\times\Acy(\sfT/\sfS)\quad\text{and}\quad
\Dl(\sfT)\xra{\sim}\Dl(\sfS)\times\Dl(\sfT/\sfS).\]
\end{proposition}

The proof is based on the following lemma.
 
\begin{lemma}\label{le:conservative}
  There exists an exact localization functor $L\col \sfT\to \sfT$ with
  $\Ker L=\sfS$, and for each object $X$ in $\sfT$ an exact triangle \[\gam
  X\to X\to L X\to\quad\text{with}\quad \gam
  X\cong\gam\one\otimes X \quad\text{and}\quad L X\cong L\one\otimes X\]
\end{lemma}
\begin{proof}
  The assumption on $\sfS$ to be well generated implies that there
  exists an exact localization functor $L\col \sfT\to \sfT$ with $\Ker
  L=\sfS$, by Proposition~\ref{prop:loc}. This functor factors through
  the quotient functor $F\col\sfT\to\sfT/\sfS$ via a functor
  $G\col \sfT/\sfS\to \sfT$, which is a right adjoint of $F$ and
  induces an equivalence $\sfT/\sfS\xra{\sim}\sfS^\perp$. 

  Completing for each object $X$ in $\sfT$ the natural morphism $X\to
  L X$ yields a functorial exact triangle
\[\gam X\lto X\lto L X\lto\]
with $\gam X\in\sfS$. Now apply $-\otimes X$ to the localization
triangle $\gam \one\to \one\to L\one\to$. Then $\gam\one\otimes X$
belongs to $\sfS$, while $L\one\otimes X$ belongs to
$\sfS^\perp$. Thus $LX\cong L\one\otimes X$ and $\gam
X\cong\gam\one\otimes X$.  Note that  $\gam$ provides a
right adjoint of the inclusion $\sfS\to\sfT$.
\end{proof}

\begin{proof}[Proof of Proposition~\ref{pr:conservative}]  
We keep the notation from Lemma~\ref{le:conservative} and its proof.

Given $X\in\sfS$ and $Y\in\sfT$, we have $X\otimes Y=0$ iff $X\otimes\gam Y=0$. Thus the inclusion is conservative. The identity $\acy(\gam Y)=\acy(Y)\cap\sfS$ implies that $\gam$ is conservative.

By a similar argument, the inclusion $\sfS^\perp\to\sfT$ and its left adjoint $L\col\sfT\to\sfS^\perp$ are conservative. The composite $E\col\sfS^\perp\to\sfT\to\sfT/\sfS$ is an equivalence of tensor triangulated categories, hence conservative. The composite of $E$ with $L$ is isomorphic to $F$, so the latter is conservative. The composite of $E^{-1}$ with the inclusion $\sfS^\perp\to\sfT$ is isomorphic to $G$. Thus $G$ is conservative.

The isomorphism  $\Acy(\sfT)\to\Acy(\sfS)\times\Acy(\sfT/\sfS)$ sends
$\acy(X)$ to $(\acy(\gam X),\acy(FX))$, and its inverse
sends $(\acy(U),\acy(V))$ to $\acy(U)\vee\acy(GV)$. Note that $X$ in
$\sfT$ is Bousfield idempotent if and only if $\gam X$ and $FX$ are Bousfield idempotent.
\end{proof}

\begin{corollary}\label{co:recollement}
Let $\sfS_1=\tloc(X_1)$ and $\sfS_2=\tloc(X_2)$ be tensor closed
localizing subcategories of $\sfT$ such that $\sfS_1^\perp$ and $\sfS_2^\perp$ are tensor
closed. Set $\sfS=\tloc(X_1\amalg X_2)$. Then the
Bousfield lattice $\Acy(\sfT)$ admits the following decomposition:
\[\Acy(\sfT) \cong\Acy(\sfT/\sfS)\times\Acy(\sfS_1/\sfS_1\cap\sfS_2)
  \times\Acy(\sfS_2/\sfS_1\cap\sfS_2)\times\Acy(\sfS_1\cap\sfS_2)\]
\end{corollary}
\begin{proof}
From Proposition~\ref{pr:conservative} one gets
\begin{align*}
  \Acy(\sfT) &\cong\Acy(\sfT/\sfS)\times\Acy(\sfS)\\
  &\cong\Acy(\sfT/\sfS)\times\Acy(\sfS/\sfS_1\cap\sfS_2)\times\Acy(\sfS_1\cap\sfS_2)\\
  &\cong\Acy(\sfT/\sfS)\times\Acy(\sfS_1/\sfS_1\cap\sfS_2)
  \times\Acy(\sfS_2/\sfS_1\cap\sfS_2)\times\Acy(\sfS_1\cap\sfS_2).
\end{align*}
For the last isomorphism one uses the pair of standard isomorphisms
\[\sfS_1/\sfS_1\cap\sfS_2\xra{\sim}\sfS/\sfS_2
\xra{\sim}(\sfS/\sfS_1\cap\sfS_2)/(\sfS_2/\sfS_1\cap\sfS_2);\] see
\cite[II.2.3]{Verdier:1996a}.  
\end{proof}

The decomposition of $ \Acy(\sfT)$ is illustrated by the following
Hasse type diagram. Given any localizing subcategory
$\sfS\subseteq\sfT$, we write $\gam_\sfS$ for a right adjoint of the
inclusion functor $\sfS\to\sfT$, assuming that it exists.

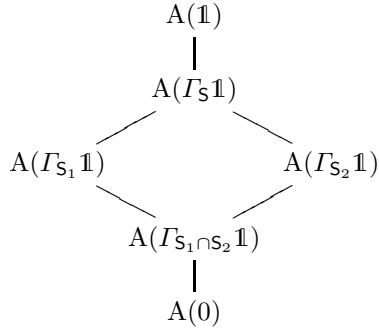
\[
\xymatrixrowsep{1pc}
\xymatrixcolsep{.2pc}
\xymatrix{
&\acy(\one) \ar@{-}[d]\\
&\acy(\gam_\sfS\one) \ar@{-}[dl]\ar@{-}[dr]\\
\acy(\gam_{\sfS_1}\one)
\ar@{-}[dr]&&\acy(\gam_{\sfS_2}\one)\ar@{-}[dl]\\
&\acy(\gam_{\sfS_1\cap\sfS_2}\one) \ar@{-}[d]\\
&\acy(0)
}
\]

\begin{example}
\label{ex:conservative}
Let $\sfT$ be a compactly generated tensor triangulated category with an action of a graded-commutative noetherian ring $R$. Denote for each $\fp\in\Spec R$ by $\gam_\fp\sfT$ the essential image of the functor $\gam_\fp\col\sfT\to \sfT$.  From the construction of $\gam_\fp$ in \cite[\S5]{\bik:2008a} and Proposition~\ref{pr:conservative}, it follows that the functors $\gam_\fp\col\sfT\to\gam_\fp\sfT$ and the inclusion $\gam_\fp\sfT\to\sfT$ are tensor triangulated and conservative. Thus these functors induce isomorphisms
\[
\Acy(\sfT)\xra{\sim}\prod_{\fp\in\Spec R}\Acy(\gam_\fp\sfT)
\quad\text{and}\quad
\Dl(\sfT)\xra{\sim}\prod_{\fp\in\Spec R}\Dl(\gam_\fp\sfT)\,.
\] 
The inverse maps send $(\acy(X_\fp))_\fp$ to $\acy(\coprod_\fp X_\fp)$; see \cite[Theorem~7.2]{\bik:2009a}. When $\sfT$ is stratified by $R$, this yields the isomorphism $\Acy(\sfT)\xra{\sim}\mathbf 2^{\supp_R(\sfT)}$ from Corollary~\ref{cor:bousfield-lattice}.
\end{example}

\begin{example}
\label{ex:noeth-scheme}
Fix a separated noetherian scheme $(X,\mathcal O_X)$ and consider the
derived category $\sfD(\Qcoh X)$ of the category of quasi-coherent
$\mathcal O_X$-modules.  This is a compactly generated tensor
triangulated category. The Bousfield lattice of $\sfD(\Qcoh X)$ is
isomorphic to the lattice of subsets of $X$.  For the affine case,
this assertion follows from Neeman's work \cite{Neeman:1992a}; see
Example~\ref{ex:ca}. The general case then follows, using
Corollary~\ref{co:recollement}; see also \cite{ASS:2004a}.  The model
for this is Gabriel's analysis of the abelian category of
quasi-coherent $\mathcal O_X$-modules \cite[Chap.~VI]{Gabriel:1962a}.
\end{example}

\section{Support for compact objects}
\label{sec:Compact objects}
Let $(\sfT,\otimes,\one)$ be a compactly generated tensor triangulated
category, as in Section~\ref{se:stratification}.  The full subcategory
$\sfT^c$ consisting of all compact objects is a skeletally small
tensor triangulated category.  Let $\Th(\sfT^c)$ denote the set of
thick subcategories of $\sfT^c$ that are tensor closed. This set is
partially ordered by inclusion and is a complete lattice, since for
any set of elements $\sfC_i$ in $\Th(\sfT^c)$ there is an equality
\[
\bigwedge_i\sfC_i=\bigcap_i\sfC_i\,.
\] 
In this section we discuss the support of objects in $\sfT^c$, using the spectrum associated with $\Th(\sfT^c)$.  Then we relate this support to the structure of the Bousfield lattice of the ambient category $\sfT$. We begin by recalling pertinent facts about
compactly generated lattices.

\subsection*{Compact generation}
Fix a complete lattice $\Lambda$.  An element $a\in\Lambda$ is \emph{compact} if $a\le\bigvee_{i\in I}b_i$ implies $a\le\bigvee_{i\in J}b_i$ for some finite subset $J\subseteq I$.  We write $\Lambda^c$ for the partially ordered subset of compact elements in $\Lambda$. Note that $a,b\in\Lambda^c$ implies $a\vee b\in\Lambda^c$. The lattice $\Lambda$ is \emph{compactly generated} if every element in $\Lambda$ is the supremum of compact elements.

Any compactly generated lattice  is essentially determined by its subset of compact elements. To explain this, we need to introduce the ideal completion of a partially ordered set.

Let $\Gamma$ be a partially ordered set having an infimum, and suppose
that every finite subset has a supremum.  A non-empty subset
$I\subseteq \Gamma$ is an \emph{ideal} of $\Gamma$ if for all $a,b\in
\Gamma$
\begin{enumerate}
\item $a\leq b$ and $b\in I$ imply $a\in I$, and
\item $a,b\in I$ implies $a\vee b\in I$.
\end{enumerate}
Given $a\in \Gamma$, let $I(a)=\{x\in \Gamma\mid x\leq a\}$ denote the
\emph{principal ideal} generated by $a$. The set $\widehat \Gamma$ of
all ideals of $\Gamma$ is called the \emph{ideal
  completion}\footnote{Viewing a partially ordered set as a category,
  the ideal completion of $\Gamma$ is nothing but the Ind-completion of
  $\Gamma$, that is, the category of functors $\Gamma^\op\to\Sets$
  that are filtered colimits of representable functors.} of
$\Gamma$. This set is partially ordered by inclusion and in fact a
compactly generated complete lattice. The map $\Gamma\to\widehat
\Gamma$ sending $a\in \Gamma$ to $I(a)$ identifies $\Gamma$ with
${\widehat \Gamma}^c$.

\begin{lemma}\label{le:compact}
  Let $\Lambda$ be a compactly generated complete lattice. Then the
  map
  \[\Lambda\lto \widehat{\Lambda^c},\quad a \mapsto I(a)\cap
  \Lambda^c=\{x\in \Lambda\mid x\leq
a\text{ and }x\text{ compact}\},\] is a lattice isomorphism.
\end{lemma}
\begin{proof}
The inverse map sends an ideal $I\in \widehat{\Lambda^c}$ to its supremum in $\Lambda$.
\end{proof}

\subsection*{Coherent frames}

A frame is called \emph{coherent} if it is compactly generated and the
compact elements form a sublattice containing $1$
\cite{Johnstone:1982}. The correspondence from
Proposition~\ref{pr:frames} between frames and topological spaces
identifies the coherent frames with the spaces that are spectral.
Following Hochster \cite{Hochster:1969a}, a topological space is
called \emph{spectral} if it is $T_0$ and quasi-compact, the
quasi-compact open subsets are closed under finite intersections and
form an open basis, and every non-empty irreducible closed subset has
a generic point.

The following characterization of the coherent frames is well-known;
we sketch the proof for the convenience of the reader. 

\begin{proposition}
\label{pr:compact}
Let $\Lambda$ be a  complete lattice. The conditions below are equivalent:
\begin{enumerate}
\item $\Lambda$ is a coherent frame.
\item $\Lambda$ is a frame, has enough points, and the associated
  spectrum is spectral.
\item $\Lambda$ is the ideal completion of a distributive lattice
  having $0$ and $1$.
\item $\Lambda$ is compactly generated, distributive, and the compact
  elements form a sublattice  containing $1$.
\end{enumerate}
\end{proposition}
\begin{proof}
(1) $\Rightarrow$ (2): Given elements $a,b\in\Lambda$ with $b\not\le
a$, there exists $c\in\Lambda^c$ such that $c\not\le a$ and $c\le b$.
The set $\{x\in\Lambda\mid a\le x,\,c\not\le x\}$ has a maximal
element (using Zorn's lemma) which is prime. Thus $\Lambda$ is a frame
having enough points.

The sets $U(a)$ with $a\in\Lambda^c$ are precisely the quasi-compact
open subsets of $\Sp(\Lambda)$; they are closed under finite
intersections and $\Sp(\Lambda)$ is of this form since $\Lambda^c$ is
a sublattice of $\Lambda$. Given an irreducible subset
$X\subseteq\Sp(\Lambda)$, the element $\bigwedge_{p \in X}p$ is prime
and therefore a generic point of $X$. Thus $\Sp(\Lambda)$ is spectral.

  (2) $\Rightarrow$ (3): The lattice $\Lambda$ is isomorphic to the
  lattice of open subsets of $\Sp(\Lambda)$, since $\Lambda$ has
  enough points. The fact that $\Sp(\Lambda)$ is spectral means that
  $\Lambda$ is isomorphic to the ideal completion of the lattice of
  quasi-compact open subsets $\Gamma=\mathcal O(\Sp(\Lambda))^c$. The
  required properties of $\Gamma$ are easily checked.

  (3) $\Rightarrow$ (4): Suppose that $\Lambda=\widehat\Gamma$ for
  some distributive lattice $\Gamma$.  The principal ideals are the
  compact elements in $\Lambda$. Therefore $\Lambda$ is compactly
  generated. The maximal element $1\in\Lambda$ is compact since
  $\Gamma$ has a maximal element. The infimum of two compact elements
  is again compact, since $I(a)\wedge I(b)=I(a\wedge
  b)$. Distributivity of $\Lambda$ follows from the distributivity of
  $\Gamma$, using that $I\vee J=\{a\vee b\mid a\in I,\,b\in J\}$.

  (4) $\Rightarrow$ (1): Compact generation of $\Lambda$ implies for
  every $a\in\Lambda$ and every directed subset $B\subseteq\Lambda$
  \[a\wedge(\bigvee_{b\in B}b)=\bigvee_{b\in B}(a\wedge b).\] Using
  that $\Lambda$ is distributive, this identity extends to arbitrary
  subsets $B\subseteq\Lambda$, since
\[\bigvee_{b\in B}b=\bigvee_{B'\subseteq B \atop \text{finite}}
\bigvee_{b\in B'}b\] 
and the finite subsets of $B$ form a directed set. Thus $\Lambda$ is a coherent frame.
\end{proof}

Any compactly generated lattice is determined by its subset of compact
elements, by Lemma~\ref{le:compact}. Combining this fact with
Propositions~\ref{pr:frames} and \ref{pr:compact} yields the following
correspondence, which is another incarnation of Stone duality.

\begin{corollary}\label{co:compact_frames}
  \pushQED{\qed} Taking a distributive lattice to the spectrum of its
  ideal completion induces a bijective correspondence:
\begin{equation*}
\left\{
\begin{gathered}
  \text{distributive lattices}\\ \text{with $0$ and $1$}
\end{gathered}\;
\right\}
\xymatrix@C=3pc{ \ar@{<->}[r]^-{\scriptscriptstyle{1-1}} &} \left\{
\begin{gathered}
  \text{spectral}\\ \text{topological spaces}
\end{gathered}\;
\right\}
\end{equation*}
The inverse takes a spectral space to its lattice of quasi-compact open subsets. \qedhere
\end{corollary}

\subsection*{Hochster duality}

Given a spectral topological space $X$, one defines its \emph{Hochster
  dual} space $X^*$ by taking the same points and turning the complements of
quasi-compact open sets of $X$ into basic open sets of $X^*$. Hochster
proved that $X^*$ is spectral and that $(X^*)^*=X$; see
\cite[Proposition~8]{Hochster:1969a}.

Corollary~\ref{co:compact_frames} explains this
duality: If a spectral space $X$ corresponds to a distributive lattice
$\Lambda$, then $X^*$ corresponds to the opposite lattice
$\Lambda^\op$.

\subsection*{The lattice of thick subcategories}
Let now $\sfT$ be a compactly generated tensor triangulated category,
as in Section~\ref{se:stratification}. We establish the lattice
theoretic properties of $\Th(\sfT^c)$, and use them to define a notion
of support for objects in $\sfT^c$.  The following map is our basic
tool:
\[
f\col\Th(\sfT^c)\lto\Acy(\sfT),\qquad\sfC
\mapsto\bigvee_{X\in\sfC}\acy(X)=\acy(\coprod_{X\in\sfC}X)\,.
\]
Observe that $f(\sfC)=\sfC^\perp$ since the objects of $\sfC$ are
strongly dualizing.

\begin{lemma}
\label{le:T^c}
The map $f$ is injective and its image is contained in $\Dl(\sfT)$. Moreover,
\[
f(\sfC\cap\sfD)=f(\sfC)\wedge f(\sfD)\quad\text{and}\quad f(\bigvee_i\sfC_i)=\bigvee_i f(\sfC_i)
\] 
for every set of elements $\sfC,\sfD,\sfC_i$ in $\Th(\sfT^c)$.
\end{lemma}

\begin{proof}
A compact object $X$ is strongly dualizable, and therefore a retract of $X\otimes X^\vee\otimes X$; see \cite[Proposition~III.1.2]{Lewis:1986a}. Hence $X$ is Bousfield idempotent. Using the properties of $\Dl(\sfT)$ from Proposition~\ref{pr:D(T)}, it follows that $f(\sfC)$ belongs to $\Dl(\sfT)$ for all $\sfC$ in $\Th(\sfT^c)$.

Next we compute $ f(\sfC\cap\sfD)$: 
\begin{align*}
  f(\sfC\cap\sfD)&\le f(\sfC)\wedge f(\sfD)=
  \big(\bigvee_{X\in\sfC}\acy(X)\big)\wedge
  \big(\bigvee_{Y\in\sfD}\acy(Y)\big)\\&=\bigvee_{(X,Y)\in\sfC\times\sfD}
  \big(\acy(X)\wedge\acy(Y)\big)=\bigvee_{(X,Y)\in\sfC\times\sfD}
  \acy(X\otimes Y)\\&\le\bigvee_{Z\in\sfC\cap\sfD}\acy(Z)=f(\sfC\cap\sfD).
\end{align*}

The identity for $f(\bigvee_i\sfC_i)$ is clear.

It remains to verify the injectivity of $f$. Since $\sfC$ consists of compact objects
\[
\tloc(\sfC)={^\perp(\tloc(\sfC)^\perp)}\quad\text{and}\quad \sfC=\tloc(\sfC)\cap\sfT^c\,. 
\]
Indeed, both equalities are well-known; the first one can be easily deduced from Lemma~\ref{le:conservative}; for the second, see \cite[Lemma 2.2]{Neeman:1992b}. Since $f(\sfC)=\sfC^\perp=\tloc(\sfC)^\perp$, it follows that $\sfC={^\perp f(\sfC)}\cap\sfT^c$, whence that $f$ is injective.
\end{proof}

Let $\sfC\subseteq\sfT^c$ be a tensor closed thick subcategory and
$\tloc(\sfC)$ the localizing subcategory generated by $\sfC$.  Recall
from Lemma~\ref{le:conservative} that there exists a localization
functor $L_\sfC\col\sfT\to\sfT$ with kernel $\tloc(\sfC)$, and for
each object $X$ in $\sfT$ an exact triangle \[\gam_\sfC X\to X\to
L_\sfC X\to\quad\text{with}\quad \gam_\sfC X\cong\gam_\sfC\one\otimes
X \quad\text{and}\quad L_\sfC X\cong L_\sfC\one\otimes X;\] see also
\cite[Theorem~3.3.3]{Hovey/Palmieri/Strickland:1997a}.  In particular,
$\tloc(\sfC)= \tloc(\gam_\sfC\one)$. This leads to the following
alternative description of the map $f\col\Th(\sfT^c)\to\Acy(\sfT)$.

\begin{lemma} 
Let $X\in\sfT$ such that $\tloc(\sfC)=\tloc(X)$. Then
$f(\sfC)=\acy(X)$. In particular, $f(\sfC)= \acy(\gam_\sfC\one)$.\qed
\end{lemma}
A consequence is the fact that the elements of the form $f(\sfC)$ have
a complement in $\Acy(\sfT)$. More precisely,
\[\acy(\gam_\sfC\one)\wedge \acy(L_\sfC\one)= 0\quad\text{and}\quad
\acy(\gam_\sfC\one)\vee \acy(L_\sfC\one)= 1.\]

The following result establishes the basic properties of the lattice
$\Th(\sfT^c)$.  The ambient category $\sfT$ is used in our proof, but
it is not essential; see \cite{Buan/Krause/Solberg:2007a} for more
general results.  In what follows for any object $X$ in $\sfT$, we
write $\Thick(X)$ for the smallest tensor closed thick subcategory of
$\sfT$ containing $X$.

\begin{theorem}
\label{th:spec}
  Let $\sfT$ be a compactly generated tensor triangulated
  category. Then the lattice $\Th(\sfT^c)$ of tensor closed thick
  subcategories of $\sfT^c$ is a coherent frame. It is isomorphic
  to the sublattice of the Bousfield lattice $\Acy(\sfT)$ consisting of the
  elements $\bigvee_{X\in\sfC}\acy(X)$ with $\sfC\subseteq\sfT^c$.
\end{theorem}

\begin{proof}
Lemma~\ref{le:T^c} provides the embedding of $\Th(\sfT^c)$ into $\Acy(\sfT)$.  In order to show that $\Th(\sfT^c)$ is a coherent frame, it suffices to verify that it is compactly generated, distributive, and that the
  compact elements form a sublattice; see Proposition~\ref{pr:compact}.

  For any object $X$ in $\sfT^c$, the category $\Thick(X)$ is a
  compact element in $\Th(\sfT^c)$. This is clear, since
  $\Thick(X)\le\bigvee_i\sfC_i$ if and only if $X\in\bigvee_i\sfC_i$,
  and keeping in mind the explicit construction of $\bigvee_i\sfC_i$
  from $\bigcup_i\sfC_i$ by taking cones of morphisms, suspensions
  etc.  Also, $\sfC=\bigvee_{X\in\sfC}\Thick(X)$ for each $\sfC$ in
  $\Th(\sfT^c)$. Thus $\Th(\sfT^c)$ is compactly generated. The
  element $1\in\Th(\sfT^c)$ is compact since
  $\sfT^c=\Thick(\one)$. Given two objects $X,Y$ in $\sfT^c$, we
  have 
\[
\Thick(X)\cap\Thick(Y)=\Thick(X\otimes Y),
\] 
since there are equalities
 \begin{align*}
   f(\Thick(X)\cap\Thick(Y))&=f(\Thick(X))\wedge
   f(\Thick(Y))\\ &=\acy(X)\wedge \acy(Y)\\& = \acy(X\otimes
   Y)\\ &=f(\Thick(X\otimes Y)).
\end{align*}
Thus the compact elements form a sublattice.  The distributivity of
$\Th(\sfT^c)$ follows from the distributivity of $\Dl(\sfT)$, using
Lemma~\ref{le:T^c}.
\end{proof}

\begin{example}\label{ex:thick}
  Let $\sfT$ be a compactly generated tensor triangulated category
  stratified by the action of a graded-commutative noetherian ring
  $R$. Suppose also that the graded endomorphism ring of each compact
  object is finitely generated over $R$.  Consider the set
  $\supp_R(\sfT)$ endowed with the Hochster dual of the Zariski
  topology; thus a subset of $\supp_R(\sfT)$ is open if it is
  specialization closed.  Then the map sending $\sfC$ to
  $\bigcup_{X\in\sfC}\supp_R(X)$ induces an isomorphism
  $\Th(\sfT^c)\xra{\sim}\mathcal O(\supp_R(\sfT))$; see
  \cite[Theorem~6.1]{\bik:2009a}.  This isomorphism induces a
  homeomorphism $\supp_R(\sfT)\xra{\sim}\Sp(\Th(\sfT^c))$; it sends
  $\fp$ in $\supp_R(\sfT)$ to
  $\{X\in\sfT^c\mid\fp\not\in\supp_R(X)\}$.
\end{example}

\subsection*{Support}
We write $\Sp(\sfT^c)=\Sp(\Th(\sfT^c))$ and for each object $X$ in
$\sfT^c$ let
\begin{align*}
  \supp_{\sfT^c}(X)  &=\{\sfP\in\Sp(\sfT^c)\mid X\not\in \sfP\}\\
&=\{\sfP\in\Sp(\sfT^c)\mid
  \Thick(X)\not\subseteq \sfP\}.
\end{align*}
Note that $ \supp_{\sfT^c}(X)$ is an open subset of $\Sp(\sfT^c)$
which is quasi-compact.  For a tensor closed thick subcategory
$\sfC\subseteq\sfT^c$, we write
\begin{align*}
  \supp_{\sfT^c}(\sfC)&=\bigcup_{X\in\sfC}\supp_{\sfT^c}(X)\\
  &=\{\sfP\in\Sp(\sfT^c)\mid\sfC\not\subseteq
  \sfP\}.
\end{align*}

Using this notation, Theorem~\ref{th:spec} has the following
consequence.

\begin{corollary}\label{co:spec}
  The spectrum $\Sp(\sfT^c)$ is a spectral topological space.  The map
  that assigns to each object $X\in\sfT^c$ its support
  $\supp_{\sfT^c}(X)$ induces an inclusion preserving bijection
  between the set of tensor closed thick subcategories of $\sfT^c$ and
  the set of open subsets of $\Sp(\sfT^c)$.
\end{corollary}
\begin{proof}
  The spectrum of a coherent frame is spectral by
  Proposition~\ref{pr:compact}. The map sending a tensor closed thick
  subcategory to an open subset is precisely the map
  \eqref{eq:adjoint} from Stone duality, which is bijective by
  Proposition~\ref{pr:frames}.
\end{proof}

The following result connects the support in $\sfT^c$ with the one
defined in $\sfT$ in terms of the Bousfield lattice of $\sfT$.

\begin{proposition}\label{pr:compact_support}
  The map $\supp_{\sfT^c}(-)$ has the following properties:
\begin{enumerate}
\item The map $f\col\Th(\sfT^c)\to \Dl(\sfT)$ induces a continuous
  map $\Sp(f)\col\Sp(\sfT)\to\Sp(\sfT^c)$ such that for each object
  $X$ in $\sfT^c$ there is an equality
\[
\supp_\sfT(X)=\Sp(f)^{-1}(\supp_{\sfT^c}(X))\,.
\]
\item $\acy(X)\leq \acy(Y) \iff \supp_{\sfT^c}(X)\subseteq \supp_{\sfT^c}(Y)$, for all $X,Y$ in $\sfT^c$.
\item $\supp_{\sfT^c}(X\otimes Y)=\supp_{\sfT^c}(X)\cap \supp_{\sfT^c}(Y)$, for all $X,Y$ in $\sfT^c$.
\end{enumerate}
\end{proposition}

\begin{proof}
  We apply Lemma~\ref{le:T^c} which lists the properties of $f$.

(1) The map $f$ is a morphism of frames and yields therefore a continuous map between the associated spectra. In particular, the identity for $\supp_\sfT(X)$ follows from equation \eqref{eq:open}.

(2) Given  objects $X,Y$ in $\sfT^c$, we have
\[
\acy(X)\leq \acy(Y) \iff\Thick(X)\subseteq\Thick(Y) \iff \supp_{\sfT^c}(X)\subseteq \supp_{\sfT^c}(Y)\,.
\] 
The first equivalence follows from Lemma~\ref{le:T^c}, while the second is a consequence of the fact that the frame $\Th(\sfT^c)$ has enough points, by Theorem~\ref{th:spec}.

(3) This follows from (2), since $\acy(X)\wedge\acy(Y)=\acy(X\otimes Y)$, by Proposition~\ref{pr:D(T)}.
\end{proof}

\begin{corollary}\label{co:cap}
For any pair of tensor closed thick subcategories $\sfC,\sfD$ of $\sfT^c$,
\[
\sfC\cap\sfD=\Thick(\{X\otimes Y\mid X\in\sfC,\,Y\in\sfD\})\,.
\] 
Therefore a tensor closed thick subcategory $\sfP\subsetneq\sfT^c$ is
prime in $\Th(\sfT^c)$ if and only if $X\otimes
Y\in\sfP$ implies $X\in\sfP$ or $Y\in\sfP$, for all objects $X,Y$ in
$\sfT^c$.
\end{corollary}

\begin{proof}
  From the formula $\supp_{\sfT^c}(X\otimes Y)=\supp_{\sfT^c}(X)\cap
  \supp_{\sfT^c}(Y)$ it follows that $\sfC\cap\sfD$ and
  $\Thick(\{X\otimes Y\mid X\in\sfC,\,Y\in\sfD\}$ have the same
  support. Thus they coincide by Corollary~\ref{co:spec}. The second
  assertion is an immediate consequence.
\end{proof}

\begin{remark}
  Let $\sfT$ be a compactly generated tensor triangulated category
  that is stratified via the action of a graded-commutative noetherian
  ring. Suppose also that the graded endomorphism ring of each compact
  object is finitely generated over $R$. The homeomorphisms
\[
\supp_R(\sfT)\xra{\sim}\Sp(\sfT)=\Sp(\Dl(\sfT))\quad \text{and}\quad 
\supp_R(\sfT)\xra{\sim}\Sp(\sfT^c)=\Sp(\Th(\sfT^c))
\] 
from Examples~\ref{ex:loc} and \ref{ex:thick} are compatible with $\Sp(f)$, in that, the following diagram is commutative.
\[
\xymatrix{
&\supp_R(\sfT)\ar[dl]_{\sim}\ar[dr]^\sim\\
\Sp(\sfT)\ar[rr]^-{\Sp(f)}&&\Sp(\sfT^c)}
\]
In particular, the map $\Sp(f)$ is bijective. The topology on
$\Sp(\sfT)$ is discrete, while the one on $\Sp(\sfT^c)$ usually is not.
\end{remark}

\subsection*{Stone versus Zariski topology}

Fix a commutative ring $A$. Then $\Spec A$, the set of prime ideals
of $A$ with the Zariski topology, is the prototypical example of a
spectral topological space \cite{Hochster:1969a}. One can think of the
tensor triangulated category $\sfD^\per(A)$ of perfect complexes over
$A$ as a categorification of $\Spec A$, because the space $\Sp
(\sfD^\per(A))$ endowed with the Stone topology is homeomorphic to the
Hochster dual of $\Spec A$; see the Example~\ref{ex:thomason} below,
which shows that the formal notion of support for tensor triangulated
categories is equivalent to the familiar notion from algebraic
geometry.\footnote{For a lattice theoretic analysis of this example,
  see \cite{Buan/Krause/Solberg:2007a}. In \cite[p.~1442]{Rota:1997a},
  Rota writes: `To this day lattice theory has not made much of a dent
  in the sect of algebraic geometers; if it ever does, it will
  contribute new insights'.}

\begin{example}
\label{ex:thomason}
Fix a quasi-compact and quasi-separated scheme $(X,\mathcal O_X)$;
every noetherian scheme has these properties. The complexes of
$\mathcal O_X$-modules with quasi-coherent cohomology form a compactly
generated tensor triangulated category $\sfT=\sfD_{\mathrm{Qcoh}}(X)$,
and its category of compact objects $\sfT^c$ identifies with the
category $\sfD^\per (X)$ of perfect complexes \cite{Neeman:1996a}. In
\cite{Thomason:1997}, Thomason classified the tensor closed thick
subcategories of $\sfD^\per (X)$ using the following notion of
support. For a complex $x\in\sfD^\per(X)$, write
\[
\supp_X(x)=\{P\in X\mid x_P\neq 0\}\,.
\]
Then the assignments
\[
\sfD^\per(X)\supseteq
\sfC\mapsto\bigcup_{x\in\sfC}\supp_X(x)\quad\text{and}\quad X\supseteq
Y\mapsto \{x\in\sfD^\per(X)\mid\supp_X(x)\subseteq Y\}
\]
induce bijections between
\begin{enumerate} 
\item the set of all tensor closed thick subcategories of $\sfD^\per(X)$, and
\item the set of all subsets $Y\subseteq X$ of the form
$Y=\bigcup_{i\in \Omega} Y_i$ with quasi-compact open complement
$X\setminus Y_i$ for all $i\in \Omega$.
\end{enumerate}
\end{example}

The example above also illustrates the different topologies that are in use in describing the spectrum of a tensor triangulated category. Thus we return to the diagram from the introduction and can now
explain its commutativity.

Fix a compactly generated tensor triangulated category $\sfT$.  It
follows from Corollary~\ref{co:cap} that the set $\Sp (\sfT^c)$ of
prime elements of $\Th(\sfT^c)$ coincides with the spectrum of prime
ideals of $\sfT^c$ defined by Balmer in \cite{Balmer:2005a}. Observe
that $\Th(\sfT^c)$ has been identified with a sublattice of the
Bousfield lattice in Theorem~\ref{th:spec}.  Using the Stone topology,
the support $\supp_{\sfT^c}(X)$ of a compact object $X$ is a
quasi-compact open subset of $\Sp(\sfT^c)$. Hochster duality
\cite{Hochster:1969a} turns this into a closed set in the Zariski
topology, which is used in \cite{Balmer:2005a}.

\subsection*{Stratification revisited}

Let $\sfT$ be a compactly generated tensor triangulated category.  We
propose a notion of stratification which does not involve the action
of a graded-commutative noetherian ring; instead we use the spectrum
$\Sp(\sfT^c)$. This provides the connection between the stratification
from \cite{\bik:2009a} and recent work of Balmer and Favi \cite{BF},
and Stevenson \cite{Stevenson:2011a}.

Suppose that the space $\Sp(\sfT^c)$ satisfies the descending chain
condition (dcc) on open subsets, that is, the Hochster dual is a
noetherian space. Fix a prime $\sfP$ in $\Sp(\sfT^c)$.  The following
lemma implies that there are open subsets $U,V$ such that $U\setminus
V=\{\sfP\}$.

\begin{lemma}
Let $X$ be a $T_0$-space satisfying the dcc on open
subsets. Then there exists for each $x\in X$ a pair of open subsets
$U,V$ such that $U\setminus V=\{x\}$.\qed
\end{lemma}

Using the bijection from Corollary~\ref{co:spec}, we get tensor closed
thick subcategories $\sfC,\sfD$ of $\sfT^c$ such that for any  $\sfQ$
in $\Sp(\sfT^c)$
\[\sfC\not\subseteq \sfQ,\, \sfD\subseteq\sfQ\quad\iff\quad\sfQ=\sfP.\]
Now define a functor $\gam_{\{\sfP\}}\col\sfT\to\sfT$ by setting
$\gam_{\{\sfP\}}=L_\sfD\gam_\sfC$. Here, $L_\sfD$ denotes the
localization functor with kernel $\tloc(\sfD)$, and $\gam_\sfC$
denotes the colocalization functor with essential image $\tloc(\sfC)$;
see Lemma~\ref{le:conservative}.  The functor $\gam_{\{\sfP\}}$ is
studied in \cite[\S7]{BF}; it is the analogue of the local cohomology
functor $\gam_\fp$ for a prime ideal $\fp$ of a graded-commutative
noetherian ring $R$ acting on $\sfT$ introduced in
\cite[\S5]{\bik:2008a}. The argument given in
\cite[Theorem~6.2]{\bik:2008a} shows that the definition of
$\gam_{\{\sfP\}}$ does not depend on the choice of
$\sfC,\sfD$. Similarly, the analogue of \cite[Lemma~2.4]{\bik:2009a}
gives
\begin{equation}\label{eq:gam_p}
\gam_{\{\sfP\}}\gam_{\{\sfQ\}}=
\begin{cases}
\gam_{\{\sfP\}}& \text{if }  \sfP=\sfQ,\\ 0&      \text{otherwise}.     
\end{cases}
\end{equation}

Following \cite[\S3]{\bik:2009a} and \cite[\S6]{Stevenson:2011a}, we
say that the \emph{local-global principle} holds for $\sfT$ if for
each object $X$ in $\sfT$
\[
\tloc(X)=\tloc(\{\gam_{\{\sfP\}}X\mid \sfP\in\Sp(\sfT^c)\})\,.
\]
For example, the local-global principle holds when $\sfT$ has a model,
by \cite[Proposition~6.7]{Stevenson:2011a}, or when the dimension of
$\Sp(\sfT^c)$ is finite \cite[Corollary~3.5]{\bik:2009a}.

Following \cite[\S4]{\bik:2009a}, we say that $\sfT$ is
\emph{stratified} by $\Sp(\sfT^c)$ if 
\begin{enumerate}
\item $\Sp(\sfT^c)$, endowed with the Stone topology, satisfies the descending chain
condition on open subsets,
\item the local-global principle holds
for $\sfT$, and 
\item for each $\sfP$ in $\Sp(\sfT^c)$ there is no non-zero
tensor closed localizing subcategory $\sfS\subseteq\sfT$ which is
properly contained in $\gam_{\{\sfP\}}\sfT$.
\end{enumerate}

A specific example of a category $\sfT$ which is stratified by
$\Sp(\sfT^c)$ is the derived category $\sfD(\Qcoh X)$ of the category
of quasi-coherent $\mathcal O_X$-modules for a separated noetherian
scheme $(X,\mathcal O_X)$; see Example~\ref{ex:noeth-scheme}.  Another
example can be found in recent work of Stevenson
\cite{Stevenson:2011b} concerning certain singularity categories.

Let us end by pointing out an analogue of Example~\ref{ex:conservative}.

\begin{proposition}
Suppose that the local-global principle holds for $\sfT$. Then the
functors $\gam_{\{\sfP\}}$ induce isomorphisms
\[
\Acy(\sfT)\xra{\sim}\prod_{\sfP\in\Sp(\sfT^c)}\Acy(\gam_{\{\sfP\}}\sfT)\quad\text{and}\quad
\Dl(\sfT)\xra{\sim}\prod_{\sfP\in\Sp(\sfT^c)}\Dl(\gam_{\{\sfP\}}\sfT)\,.
\]
If $\sfT$ is stratified by $\Sp(\sfT^c)$, this yields an isomorphism
$\Acy(\sfT)\xra{\sim}\mathbf 2^{\Sp(\sfT^c)}$.
\end{proposition}
\begin{proof}
Sending $(\acy(X_\sfP))_\sfP$ to $\acy(\coprod_\sfP X_\sfP)$ provides
an inverse. This follows from the local-global principle and the
identity \eqref{eq:gam_p}.
\end{proof}

\bibliographystyle{amsplain}

\end{document}